\documentclass{article}
\usepackage{amsmath,amssymb,amsthm,graphicx,epsfig,float,url}
\usepackage[colorlinks=true,citecolor={Plum},linkcolor={Periwinkle}]{hyperref}
\usepackage{pdfsync}
\usepackage[usenames, dvipsnames]{xcolor}
\usepackage{tikz}
\usepackage{subfig}
\usepackage{bbm}
\usepackage{stmaryrd}
\usepackage{mathrsfs}  
\usepackage{caption}
\usepackage{float}
\usepackage{esint}
\usepackage[english]{babel}
\usepackage{dsfont}
\usepackage{graphicx}
\usepackage{grffile}

\usepackage[shortlabels,inline]{enumitem}

\usepackage[makeroom]{cancel}
\usetikzlibrary{patterns}
\usepackage{accents}

\usepackage{accents}

\newcommand{\R}{\textnormal{I\kern-0.21emR}}
\newcommand{\N}{\textnormal{I\kern-0.21emN}}

\renewcommand{\geq}{\geqslant}
\renewcommand{\leq}{\leqslant}

\def\e{{\varepsilon}}

\def\TT{{(0,T)\times \T}}
\def\T{{\mathbb{T}^d}}

\newcommand{\vp}{\varphi}

\newcommand{\Kt}{{\widetilde{K}}}
\renewcommand{\th}{\theta}

\allowdisplaybreaks

\def\YYint#1#2#3{{\setbox0=\hbox{$#1{#2#3}{\iint}$}
    \vcenter{\hbox{$#2#3$}}\kern-.51\wd0}}
 
\usepackage[dvipsnames]{xcolor}

\newtheorem*{theorem*}{Theorem}

\newtheorem{theorem}{Theorem}  
\newtheorem{proposition}{Proposition}

\newtheorem{definition}{Definition}
\newtheorem{lemma}{Lemma}

\theoremstyle{definition}\newtheorem{remark}{Remark}

\def\O{{\T}}
\def\n{{\nabla}}
\def\p{{\varphi}}

\usepackage{xargs}
 \usepackage[colorinlistoftodos,textsize=small]{todonotes}
 \newcommandx{\unsure}[2][1=]{\todo[linecolor=red,backgroundcolor=red!25,bordercolor=red,#1]{#2}}
 \newcommandx{\change}[2][1=]{\todo[linecolor=blue,backgroundcolor=blue!25,bordercolor=blue,#1]{#2}}
 \newcommandx{\info}[2][1=]{\todo[linecolor=green,backgroundcolor=green!25,bordercolor=green,#1]{#2}}
 \newcommandx{\improvement}[2][1=]{\todo[linecolor=yellow,backgroundcolor=yellow!25,bordercolor=yellow,#1]{#2}}
 
  \newcommandx{\biblio}[2][1=]{\todo[linecolor=blue,backgroundcolor=magenta!25,bordercolor=blue,#1]{#2}}

\newcommand{\add}[1]{{{\color{blue}#1}}}
\newcommand{\note}[1]{{\textbf{\color{red}#1}}}

\begin{document}
\title{Mean-field games for harvesting problems: \\Uniqueness, long-time behaviour and weak KAM theory}


\author{Ziad Kobeissi\footnote{ Universit\'e Paris-Saclay, CNRS, CentraleSup\'elec, Inria, Laboratoire des signaux et syst\`emes, 91190, Gif-sur-Yvette, France (\texttt{ziad.kobeissi@inria.fr})},\,  Idriss Mazari-Fouquer\footnote{CEREMADE, UMR CNRS 7534, Universit\'e Paris-Dauphine, Universit\'e PSL, Place du Mar\'echal De Lattre De Tassigny, 75775 Paris cedex 16, France, (\texttt{mazari@ceremade.dauphine.fr})},\,  Dom\`enec Ruiz-Balet\footnote{Department of Mathematics, Imperial College, London SW7 2AZ, United Kingdom (\texttt{d.ruiz-i- balet@imperial.ac.uk})}}

\maketitle

\begin{abstract}

The goal of this paper is to study a Mean Field Game (MFG) system stemming from the harvesting of resources. 
Modelling the latter through a reaction-diffusion equation and the harvesters as competing rational agents, we are led to  a non-local (in time and space) MFG system that consists of three equations, the study of which is quite delicate. The main focus of this paper is on the derivation of analytical results (\emph{e.g.} existence, uniqueness) and of long time behaviour (here, convergence to the ergodic system). We provide some explicit solutions to this ergodic system.\end{abstract}


\medskip


\paragraph{Acknowledgement:} This work was initiated during visits of D. Ruiz-Balet at CEREMADE, Paris Dauphine Universit\'e PSL, and of I. Mazari-Fouquer at the Imperial College; the hospitality of both institutions is acknowledged. These visits were made possible by the Abraham de Moivre short stage mobility grant, the KTK SprintChallenge  and a PEPS project from CNRS. I. Mazari-Fouquer was partially funded by a PSL Young Researcher Starting Grant 2023.

\section{Introduction}

\subsection{Scope of the paper}
Our goal in this paper is the derivation and analysis of a model where rational agents compete in order to harvest as much as possible from a common pool of  renewable resources. In order to fix the terminology, resources will be referred to as ``fishes", while the competing agents will be dubbed ``fishermen". This choice of interpretation is motivated by the fact that the exploitation of fisheries is a global problem in the handling and preservation of ecosystems, and can be used as a paradigmatic case of resources harvesting.  From a modelling point of view, our main motivation is to propose a new approach to the theoretical study of fishing phenomena, spatial ecology and the tragedy of the commons.
While several biological and economical contributions have tackled this issue from the applied point of view, shedding new light on our understanding of the phenomenon, the mathematical literature devoted to this issue emerged only recently. We review it in section \ref{Se:Biblio}. In this work, we provide a qualitative study of this new type of system blending ecological phenomena and economical competition. To do so, we rely on two main theories. The first one is the classical theory of reaction-diffusion equations, which are commonly used in the description of population dynamics  \cite{CantrellCosner}. These equations are used, here, to model the fishes' density $\theta$, as they can incorporate intra-specific phenomena (birth, death, interactions with the environment) and sensitivity  to external influences (typically, fishing).  Second, in order to model the behaviour of a large number of fishermen engaged in economic competition, we use the Mean Field Game framework, which allows to track down the evolution of the density $m$ of fishermen. Thus, there are two additional equations to be dealt with, one  for $m$, and one for the value function $u$ (that is, the maximal expected profit of an agent).

Consequently, we obtain a system of three coupled partial differential equations (see \eqref{Eq:MFGCompet} for details).
This system belongs to a novel class of MFG characterised by a non-local (both in time and space) coupling which, to the best of our knowledge, was first studied in our previous paper \cite{KMFRB23}.
Our main objective is to propose flexible techniques to investigate the analytical aspects of such systems and, more specifically, questions of uniqueness and longtime behaviour of the solutions. In the next section, we derive our main system, which will allow for a more in-depth discussion of our motivations. 

\subsection{The optimal control problem of an individual fisherman}
\label{subsec:optcont}
Throughout, the variable $x$ refers to the spatial location of fishes and fishermen, while $t$ denotes the time. As is common in the MFG literature, we make the stringent assumption that the spatial domain is the $d$-dimensional torus $\T$. This choice helps to avoid analytical difficulties linked to boundary conditions in MFG systems and allows us to focus on the innovative aspects of our approach. 

In the remainder of this section, we focus on the description of a simple scenario that is the building block of the forthcoming MFG system. In this scenario, a single fisherman controls his trajectory, while the rest of the fishermen population has a pre-determined evolution. The goal of this single fisherman is  to maximise his profit over a time interval $[0,T]$, where $T>0$ denotes a fixed time horizon.
We assume that the fisherman can control his velocity, which is also subject to some brownian fluctuation, so that its trajectory $(X_t)_{t>0}$ satisfies  the following dynamic:
\begin{equation}\label{Eq:Representative}
d X_t=\alpha_tdt+\sqrt{2\nu}dB_t,\end{equation}
where $\nu\geq 0$ is a viscosity coefficient, $(B_t)_{t\geq 0}$ is a standard Brownian motion and $\alpha$ is the control.

The population of fishes is modelled by a density $\theta:[0,T]\times\T\to[0,\infty)$.
At time $t$ and state $x$, the fisherman's profit per unit of time from fishing is assumed to be $\rho\star(\theta(t))(x)$ (denoted as $\rho\star\theta(t,x)$ for simplicity), where $\star$ denotes the convolution  in space and $\rho$ is a smooth, even function with a compact support, satisfying $\rho\in \mathscr C^\infty_c(\T)\,, \rho\geq 0\,, \int \rho=1$.
This modelling choice reflects a situation where the fisherman  can catch fishes within a small surrounding area (typically, using fishing nets).
The fisherman's objective is then to choose $\alpha$ to maximise his net profit or, in other words, to solve the optimisation problem:
 \begin{equation}\label{Eq:OptCont}
\max_{\alpha}J(\alpha):=\mathbb E\left[\int_0^T (\rho \star\theta)\left(t,X_t\right)dt-\int_0^TL(X_t,\alpha_t)dt\right],
 \end{equation}
where $L$ is a running cost referred to as the Lagrangian,
which we take in this paper  to be quadratic:
\[ L(x,\alpha)=\frac{\Vert \alpha\Vert^2}2.\]
We might consider more general Lagrangians;
however, in order to put forth the weak KAM framework of \cite{CardaliaguetKAM}, we shall make use of regularising properties of the HJB equation that are easier to derive when working with quadratic Hamiltonians, and, consequently, quadratic Lagrangians.

Observe that there are several underlying assumptions in this formulation: first, we assume that a single fisherman has a negligible impact on the fishes' population. Mathematically, this is interpreted as 
$\theta$ being independent of the control $\alpha$
(a condition that will not hold when the discussion shifts to MFG systems in Section \ref{subsec:MFG}). Second, we assume that the fisherman has a perfect knowledge of the density $\theta$ in all places, and at all times (even in the future).
We define the value function $u=u(t,x)$ as the maximal expected profit in the interval $[t,T]$
for an agent located in $x$ at time~$t$. In other words, $u$ is defined as 
\[ u(t,x)=\max_{\alpha}\mathbb E\left[\int_t^T (\rho\star\theta)(s,X_s)ds-\int_t^T\frac{\|\alpha_s\|^2}2ds\,\Big|\, X_t=x\right].\]
By a standard application of the Bellman principle, the function $u$  is the unique (classical if $\nu>0$, viscosity if $\nu=0$) solution of the HJB equation 
\begin{equation}\label{Eq:HJBCompet}
\begin{cases}-\partial_t u-\nu\Delta u-\frac{\Vert \n_xu\Vert^2}2=\rho\star\theta\text{ in }(0,T)\times \T\,,\\ u(T,\cdot)\equiv 0\end{cases}
\end{equation}
Given $u$ a regular solution (if $\nu>0$) or a viscosity solution (if $\nu=0$) to the latter equation, an optimal control solving the individual optimisation problem is given by
\begin{equation} 
    \alpha(t,x)=\nabla_x u(t,x).
\end{equation}
Assume now that the initial location of the fisherman is given by a random variable $X_0$ following a fixed
distribution $m_0\in\mathcal{P}(\T)$.
Denoting by $m(t)$ the law of $X_t$ at any given time $t\in[0,T]$, $m$ satisfies a Fokker-Planck-Kolmogorov (FPK) equation
\begin{equation}\label{Eq:FPKCompet}
\begin{cases}
\partial_tm-\nu\Delta m+\nabla \cdot\left(m \n_x u \right)=0\text{ in }(0,T)\times \T\,, \\m(0,\cdot)=m_0.\end{cases}
\end{equation}In the subsequent section, we will interpret $m$ as being both the law of a single fisherman at time $t$, but also as the distribution, at time $t$, of a fleet of fishermen initially distributed according to $m_0$ and with individual controls $\n_x u$.

\subsection{Modelling of the fishes' population}
\label{subsec:fishesdyn}
As is customary in population dynamics \cite{CantrellCosner,Lam_2022,Murray}, we assume that the population of fishes  is subject to three main phenomena:
\begin{enumerate}
\item First, diffusion, which accounts for a brownian movement in the domain. The strength of the diffusion is quantified by a diffusivity $\sqrt{2\mu}$, where $\mu>0$ is the characteristic dispersal rate.
\item Second, a linear growth and logistic death. It should be noted that some spatial heterogeneity is taken into account by considering heterogeneous resources distributions in the domain. The resources distribution is a function $K:\T\to \R$. The zones where $K$ is positive are favorable to the population, while the zones where $K$ is non-positive are lethal. The growth term in the model is of the form $K\theta$. 
    An additional death term is implemented \emph{via} a standard Malthusian non-linearity $-\theta^2$, thus accounting for crowding effects (or, alternatively, symmetric competition within the population). 
\item Finally, fishing.
        As already explained, we assume that each fisherman is able to fish within a small surrounding area, quantified through a convolution kernel $\rho$, and  his catch per unit of time is proportional to  $\rho\star\theta$. This and the fact that $\rho$ is assumed to be even lead to a fishing term of the form $-\e(\rho\star m)\theta$ where $\e$ is called the fishing pressure or capacity, and we recall that $m$ is the density of fishermen.
\end{enumerate} 
Observe that our model, although relevant given its proven ability to capture paradigmatic qualitative phenomena in population dynamics, does not incorporate other important effects. Typically, we do not cover models featuring the Allee effect, which is used in situations where, if the population is initially too small, it eventually dies out. We considered such models and their interplay with MFG problems in our previous contribution \cite{KMFRB23}, where we derived some explicit travelling waves solutions that were then used to provide instances of the tragedy of the commons. 

To underline the dependency of the resulting fishes' density on $\e$ and $m$ we denote it by $\theta_{\e,m}$.
The equation on the population density then reads
\begin{equation}\label{Eq:Main}
\begin{cases}
\partial_t \theta_{\e,m}-\mu \Delta \theta_{\e,m}=\theta_{\e,m}(K-\theta_{\e,m})-\e(\rho\star m)\theta_{\e,m}&\text{ in }(0,T)\times \T\,,
\\{\theta_{\e,m}(0,\cdot)=\theta_0}&\text{ in }\T\end{cases}
\end{equation}
where $\theta_0\in L^\infty(\T)\,, \theta_0\geq 0\,, \theta_0\neq 0$ is the initial distribution of fishes.

\subsection{The mean field game models}
\label{subsec:MFG}
As stated in Section \ref{subsec:fishesdyn}, the density of fishes $\theta$ can be obtained by solving \eqref{Eq:Main} knowing $m$, the density of fishermen. Conversely, in Section \ref{subsec:optcont}, for a fixed density of fishes $\theta$, we were able to characterise the distribution of a fleet of  fishermen. As is customary in MFG, this leads to a fixed point argument described as follows (where $m_0$ and $\theta_0$ are fixed):
\begin{itemize}
    \item
        A prediction is made
        on the values of $\theta$ at every time and location.
    \item
        With $\theta$ fixed as above,
        the control problem
        \eqref{Eq:OptCont} is solved from every location in supp$(m_0)$.
    \item
        Using the optimal control derived from the previous step,
        the PDE \eqref{Eq:FPKCompet} is solved
        to get the distribution of the fishermen $m(t)$
        at any time $t\in[0,T]$.
    \item
        With $m$ given in the previous step,
        a new fishes' distribution $\widetilde{\theta}$
        is computed by solving \eqref{Eq:Main}.
    \item
        A solution (which is also called a MFG equilibrium) is found  when $\widetilde{\theta}=\theta$.
\end{itemize}
\begin{remark}[Uniqueness in MFG systems]
In the MFG literature, it is customary to distinguish between the existence of solutions, which is often an easy consequence of fixed-point theorems, and their possible uniqueness. The latter  is usually much more challenging, and several  paradigms that we review in the bibliographical section of this work exist to tackle it.
We approach this issue by introducing a novel method that relies heavily on reaction-diffusion techniques.
This method represents one of the key contributions of our work, but nevertheless applies in a perturbative setting. The question of uniqueness in general remains open. Regarding the interest of having uniqueness of solutions to MFG systems, it is important to observe that, in order for an individual to select the control provided by the MFG equilibrium, he needs to predict the values of $\theta$ associated with this equilibrium. Furthermore, the existence of multiple equilibria makes it \emph{a priori} impossible to select a 
{``best''} one, or to define an {``expected''} behaviour of a rational agent. This is similar to the situations that arise when dealing with Nash equilibria. Thus, being able to get the uniqueness of equilibria provides a justification for the emergence of this equilibrium in practice.

For the statement of our main results on uniqueness, we refer to Section \ref{Se:Uniqueness}.
For a comprehensive discussion regarding uniqueness in MFG systems,
we refer to  Section \ref{Se:Biblio}.
\end{remark}

The complete Mean Field Game system thus reads
\begin{equation}\tag{$\bold{MFG}$}\label{Eq:MFGCompet}
\begin{cases}
-\partial_t u-\nu\Delta u-\frac{\Vert \n_xu\Vert^2}2=\rho\star\theta_{\e,m}&\text{ in }(0,T)\times \T\,,
\\ u(T,\cdot)\equiv 0,
\\\partial_tm-\nu\Delta m+\nabla \cdot\left(m\n_x u\right)=0&\text{ in }(0,T)\times \T\,, \\m(0,\cdot)=m_0,
\\ \partial_t \theta_{\e,m}-\mu \Delta \theta_{\e,m}=\theta_{\e,m}(K-\theta_{\e,m})-\e(\rho\star m)\theta_{\e,m}&\text{ in }(0,T)\times \T\,, 
\\ \theta_{\e,m}(0,\cdot)=\theta_0.
\end{cases}
\end{equation}
We will be interested in both the first-order case $\nu=0$ and the second order case $\nu>0$, the analysis of the latter being a necessary detour to obtain the required regularity estimates for the former, as is customary in MFG.

\begin{remark}[Notable difference with other classes of MFG]
\label{Rk:MFGclass}
Seeing $\theta_{\e,m}$ as a function of $m$  the system \eqref{Eq:MFGCompet} can \emph{a priori} be considered as an MFG system with a non-local coupling. However, unlike usual non-local MFG systems \cite{Bardi_2019,CardaliaguetLasryLionsPorretta}, the non-locality is in both space and time: $\theta_{\e,m}(t,\cdot)$ depends on the entire family $(m(s,\cdot))_{0\leq s\leq t}$. This sets \eqref{Eq:MFGCompet} apart from the classical setting of MFG systems.\end{remark}

The second part of our main results deals with
the long-time behaviour of \eqref{Eq:MFGCompet}.
Consequently, we need to introduce the associated ergodic MFG system,
which ``should" correspond to the large time asymptotics of \eqref{Eq:MFGCompet}. To make the time dependency explicit, and assuming that the solution of \eqref{Eq:MFGCompet} is unique for any $T>0$ (which will be guaranteed by Theorem \ref{Th:UniquenessMFGCompet}), let us denote the solution of \eqref{Eq:MFGCompet} by $(\theta_{\e,m}^T\,, m^T\,, u^T)$. Following the analysis of Cardaliaguet \cite{CardaliaguetKAM} it is expected that, in some ``reasonable" topologies and for some constant $\overline \lambda$, dubbed the ergodic constant, the triplet $(\theta_{\e,m}^T\,, m^T\,, u^T/{\overline \lambda })$  converges, as $T\to \infty$, to $(\overline\theta_{\e,\overline m}\,, \overline m\,, \overline u)$, where $(\overline \lambda,\overline\theta_{\e,\overline m},\overline m,\overline u)$ is the solution\footnote{See Definition ~\ref{De:SolutionMFGCompet} below for the exact meaning of ``solution".} of 
 \begin{equation}\label{Eq:ErgodicMFG}\tag{$\bold{MFG}_{\mathrm{ergo}}$}
 \begin{cases}
 \overline \lambda -\nu\Delta \overline u-\frac{\Vert \n_x\overline u\Vert^2}2=\rho\star\overline\theta_{\e,\overline m}&\text{ in }\T\,,
\\-\nu\Delta\overline m+\nabla \cdot\left(\overline m\n_x \overline u\right)=0&\text{ in }\T\,,
\\ -\mu \Delta \overline\theta_{\e,\overline m}=\overline\theta_{\e,\overline m}(K-\overline\theta_{\e,\overline m})-\e(\rho\star\overline m)\overline\theta_{\e,\overline m}&\text{ in } \T\,, 
\\ \overline\theta_{\e,m}\geq 0,\;
\overline \theta_{\e,m}\not\equiv 0,\;
 \int_{\T}\overline u=0\,, \overline m\geq 0\,, \int_\T \overline m=1.\end{cases}\end{equation}
Observe that, from the first and second line, $u$ is defined up 
to an additive constant, whence the condition $\int_{\T}\overline u=0$.
The quantity $\overline \lambda$ is the so-called ``ergodic constant''.

\subsection{Stationary solutions of diffusive logistic equations}
\label{subsec:stationary}
Our analysis will in great parts rely on the properties of stationary solutions of a reaction-diffusion equation resembling \eqref{Eq:Main}, with a general resource distribution $\Kt\in L^\infty(\T)$. We consider the following equation:
\begin{equation}\label{Eq:Main2}
\begin{cases}
\partial_t \theta-\mu\Delta \theta=\theta(\Kt-\theta)&\text{ in }(0,T)\times \T\,, 
\\ \theta(0,\cdot)=\theta_0\geq 0\,, \neq 0&\text{ in }\T.
\end{cases}
\end{equation}
From \cite{BHR}, if $\int_{\T}\Kt>0$, there exists a unique,
non-trivial stationary  solution $\overline \theta_{\Kt}$ of~\eqref{Eq:Main2}. It satisfies
\begin{equation}\label{Eq:ELD}\begin{cases}
-\mu \Delta \overline \theta_{\Kt}-\overline \theta_{\Kt}(\Kt-\overline \theta_{\Kt})=0&\text{ in }\T\,, 
\\ \overline \theta_{\Kt}\geq 0\,, \overline\theta_\Kt\neq 0\,.\end{cases}\end{equation}
The steady state $\overline\theta_{\Kt}$ is globally, exponentially stable:
for $\theta_0$ a non-zero, non-negative bounded function,
there exists a constant $C_{\Kt}>0$ that depends only on $\theta_0$,
$\int_\O {\Kt}$ and $\Vert {\Kt}\Vert_{\mathscr C^1(\O)}$ such that,
for any $t\in \R_+$,
\begin{equation}\label{Eq:EldLongTime}
    \Vert \theta(t,\cdot)-\overline\theta_{\Kt}\Vert_{L^\infty(\O)}\lesssim e^{-C_{\Kt} t} \Vert \theta_0-\overline\theta_{\Kt}\Vert_{\mathscr C^0(\T)}.
\end{equation}Furthermore, this constant can be chosen to only depend on $\int_\T \theta_0^2$ and $\Vert\theta_0\Vert_{L^\infty(\T)}$.

The eigenvalues of symmetric operators will play an important roles in the upcoming proofs. For any $V\in L^\infty(\T)$ and any $\mu>0$, introduce the first eigenvalue $\lambda_1(\mu,V)$ of the operator $-\mu\Delta-V$; it admits the variational formulation
\begin{equation}\label{Eq:Rayleigh}
\lambda_1(\mu,V)
=\min_{u\in W^{1,2}(\T)\,, \int_\T u^2=1}\mu\int_\O |\n u|^2-\int_\O V u^2.\end{equation}
It is standard to see that $\lambda_1(\mu,V)$ is simple, and that it is the only eigenvalue whose associated eigenfunction has constant sign. Now, observe that, defining $V_{\Kt}:=\Kt-\overline\theta_{\Kt}$, $\overline\theta_{\Kt}$ is an eigenfunction of $-\mu\Delta-V_{\Kt}$, with eigenvalue 0. As $\theta_{\Kt}\geq0\,,\neq 0$, we deduce that $\lambda_1(\mu,V_{\Kt})=0$.
\section{Main results}
\label{Se:Results}
Our main purpose in this paper is two-fold.
First, we want to introduce a new method to obtain {analytic properties} of solutions to \eqref{Eq:MFGCompet}, see Section \ref{Se:Uniqueness}.
The question of uniqueness is not solved fully here; we partially address this issue in Theorem \ref{Th:UniquenessMFGCompet}, using an original strategy. The sharpness of this result is then mitigated in Theorem \ref{Th:Ergodic1D}, where we demonstrate that uniqueness can be achieved outside of the setting of Theorem \ref{Th:UniquenessMFGCompet}; in particular, in a simpler one-dimensional model where explicit solutions can be constructed, and their uniqueness established.
Second, we want to derive the long-time behaviour of \eqref{Eq:MFGCompet},
by showing how one might adapt the (weak) KAM framework
developed in the context of MFG in \cite{CardaliaguetKAM}
to this new blend of MFG and reaction-diffusion equations, we refer to Section \ref{Se:KAM}.

\subsection{Assumptions}
\label{subsec:assumptions}
Before we state our main assumptions and results, let us introduce some notations.
$\mathcal M(\T)$ stands for the set of Radon measures on $\T$, endowed with the total variation norm $\Vert \cdot\Vert_{\mathcal M(\T)}$. We recall that this norm is defined as follows: for any Radon measure $\mu$, decompose it into its positive and negative parts $\mu=\mu_+-\mu_-$ where $\mu_\pm$ are non-negative Radon measures with disjoint supports. Then 
\begin{equation}\label{Eq:TotalVariationNorm} \Vert \mu\Vert_{\mathcal M(\T)}=\mu_+(\T)+\mu_-(\T).\end{equation}
The Sobolev spaces $W^{k,p}(\T)$ are defined as usual. 
For functions depending only on the space variable $x\in \T$, for any $\beta\in (0,1)$, it is understood that the space $\mathscr C^{k+\beta}(\T)$  denotes the space of $k$-times differentiable functions, with $\beta$-H\"{o}lder continuous derivatives. For functions depending on time and space, for any $\alpha,\beta \in (0,1)\,, k,j\in \N$, it is understood that the space $\mathscr C^{j+\alpha,k+\beta}(\TT)$ denotes the set of functions that are $\mathscr C^{j+\alpha}$ in time and $\mathscr C^{k+\beta}$ in space.

We are now in position to define the notion of solution that will be used throughout this article;
\begin{definition}\label{De:SolutionMFGCompet}
    A solution of the MFG system \eqref{Eq:MFGCompet}
    is a triplet $(u,m,\theta)$ such that:
    \begin{itemize}
        \item     $u\in\mathscr  C^{0,1}([0,T]\times\T)$ is a viscosity solution of \eqref{Eq:HJBCompet} if $\nu=0$, and a classical solution if $\nu>0$,
        \item  $m\in C^0([0,T];\mathcal{P}(\mathbb{T}^d))$ (for the Wasserstein distance) is a solution of \eqref{Eq:FPKCompet} in the sense of distributions if $\nu=0$, and a classical solution if $\nu>0$.
        \item $\theta\in \mathscr C^{1,2}([0,T]\times\T)$ is a classical solution of \eqref{Eq:Main}.
    \end{itemize}
\end{definition}
     It should be noted here that we are requiring some strong regularity  on $\theta$. As our purpose in this paper is to focus on qualitative properties, we postpone the exploration of lower regularity solutions to future works. In particular, the regularity required of $\theta_{\e,m}$ will lead us to make some strong assumptions on $\theta_0$ and $K$.
Namely, we will assume:
\begin{enumerate}[label=$(\bold{H_K})$]
    \item
        \label{H:K}
        $K\in \mathscr C^{2+\beta}(\T)$ for some $\beta\in(0,1)$ 
        and
        $\int_\T K>0$.
\end{enumerate}
\begin{enumerate}[label=$(\bold{H_{\theta_0}})$]
    \item
        \label{H:Theta0}
        $\theta_0=\overline{\theta}_K$
        solves 
        \eqref{Eq:ELD}
        and either:
        \begin{enumerate*}[label=$\bold{\alph*)}$]
            \item
                \label{H:Theta0a}
                $d=1$
                \hspace*{0.2cm}
                or
                \hspace*{0.2cm}
            \item
                \label{H:Theta0b}
                $\Vert K-\int_\T K\Vert_{L^\infty(\T)}\leq \delta$,
                for some $\delta>0$.
        \end{enumerate*}
\end{enumerate}
Let us comment on \ref{H:Theta0}: regarding the fact that $\theta_0=\overline\theta_K$, this is not too problematic from a modelling perspective as we can think of the fishes as having adapted to their environment before the fishing starts. It will also be important in deriving our uniqueness results. This assumption could in fact be slightly relaxed by assuming that the fishes are not ``too far" from equilibrium, but this would simply make the notations heavier.  The other two conditions on the other hand may seem more mysterious. They are linked, as will be explained throughout the proofs, to some monotonicity properties that were explored in one of our previous works \cite{Mazari_2022} devoted to the investigation of Nash equilibria for a finite number of players.

We recall that $\overline \theta_K$ is the (unique) stationary state of \eqref{Eq:ELD}. By elliptic regularity, if \eqref{H:Theta0} is satisfied, then
\[ \overline \theta_K \in \mathscr C^{2+\beta}(\T).\]

\subsection{Existence and uniqueness results}
\label{Se:Uniqueness}
Let us first state our results concerning existence of solutions.
The main theorem is the following:
\begin{theorem}\label{Th:ExistenceMFGCompet}
    Let $K$ satisfy \eqref{H:K} and $\theta_0$ be $C^{2+\beta}$.
    For any $\e>0$, for any $\nu\geq 0$,  there exists at least one solution of \eqref{Eq:MFGCompet}
    (resp. \eqref{Eq:ErgodicMFG}).
\end{theorem}
For the proof of the existence of solutions to the time-dependent system \eqref{Eq:MFGCompet},
we refer to Section \ref{subsec:exist_TD}, where we first  treat the second-order case $\nu>0$  through standard parabolic estimates. The first-order case $\nu=0$ is dealt with using a vanishing viscosity approach coupled with appropriate uniform (in the vanishing diffusivity $\nu$) estimates.
For \eqref{Eq:ErgodicMFG} in the second order case, we do not detail the proof as it is similar to the one in \cite{CardaliaguetLasryLionsPorretta}.
In the first-order case, we mimick the approach of \cite{CardaliaguetKAM},
see Section \ref{subsec:exist_KAM}. The reason we include this proof and not the one for the second-order case is that the proof of  \cite{CardaliaguetKAM} involves the weak KAM formula, which plays a role in the construction of explicit solutions, and this seems a natural place to introduce it.

While the existence of solutions follows from standard arguments, uniqueness does not. Overall, there are only very few structural assumptions that can guarantee uniqueness in MFG systems. The Lasry-Lions monotonicity condition \cite[Theorem 2.4]{LasryLions3} is probably the main one. Another standard way to overcome the possible lack of uniqueness is to make some smallness assumption, either on the time horizon $T$ or on the strength of the mean-field interaction which, in the present paper, is
quantified by $\e>0$. Such assumptions allow to use Gr\"{o}nwall-type lemmas to derive a contractivity property of the solution mapping, leading to the required uniqueness.
This ``small interaction" regime is very restrictive, as $\e$ is typically exponentially  small with respect to $T$ (\emph{i.e.} $\e\lesssim e^{-CT}$). In particular, it prohibits investigating the long-time behaviour of the system.
We also refer to section \ref{Se:Biblio} for more details on uniqueness issues in Mean Field Games.

Our uniqueness result is also shown under an assumption on the smallness of $\e$, but we obtain a smallness regime that is independent of the time horizon. Our strategy consists in showing that, under the assumption ``$\e>0$ is small enough", the Lasry-Lions monotonicity condition is satisfied. In particular, our strategy of proof creates a link between these two regimes, which are generally presented as distinct hypothesis \cite{Bardi_2019}.  This is crucial, as we later want to investigate the long-time behaviour of this system.

Our main theorem is the following:
\begin{theorem}
    \label{Th:UniquenessMFGCompet}
 There exists $\e_0>0$ and $\delta>0$ such that, if \ref{H:K} and \ref{H:Theta0} hold then, for any $\e\in (0;\e_0)$,
    there exists at most one solution to \eqref{Eq:MFGCompet} for any $T>0$
    (resp. \eqref{Eq:ErgodicMFG})
\end{theorem}
The proof of this theorem relies on the following monotonicity result, which also plays a crucial role in studying the long-time behaviour of \eqref{Eq:MFGCompet}:
\begin{proposition}\label{Pr:EpsilonMonotone}
    Assume \ref{H:K} and \ref{H:Theta0} hold
  and  let $\mu>0$. For a given $m\in\mathscr C^0(0,T;\mathcal P(\T))$ and $\e>0$, let $\theta_{\e,m}$ denotes the solution to \eqref{Eq:Main}.
    There exists $\e_0>0$ and $a>0$ such that,
    for any $\e\leq \e_0$,
    for any $T\in(0,\infty)$,
    the following monotonicity property holds:
    for any $m,m'\in \mathscr C^0([0,T];\mathcal P(\T))$,
\begin{equation}
    \label{Eq:Mono}
    \iint_{(0,T)\times \T} \Big(\rho \star \theta_{\e,m}-\rho\star\theta_{\e,m'}\Big)d (m-m')
    \leq -a\e \iint_\TT\left(\theta_{\e,m_1}-\theta_{\e,m}\right)^2.
\end{equation}
\end{proposition}
Note that simply using the fact that the right-hand side of \eqref{Eq:Mono} is negative for $m\neq m'$ yieldsthe well-known Lasry-Lions monotonicity condition, which is  sufficient to obtain uniqueness for \eqref{Eq:MFGCompet} for a fixed time horizon $T$.
However, to obtain the long-time convergence to \eqref{Eq:ErgodicMFG}, the quantified estimate \eqref{Eq:Mono} is necessary, as in \cite{CardaliaguetKAM},  see Section \ref{Se:Biblio} for more details.

Before we conclude this section, let us point out that we do not expect
our uniqueness results in Theorem~\ref{Th:UniquenessMFGCompet} to be sharp. 
Indeed, at least in dimension one and without the convolution
kernel $\rho$, we are able to prove uniqueness under milder assumptions,
as stated in the following theorem.

\begin{theorem}\label{Th:Ergodic1D}
    Let $d=1$ and identify $\mathbb T$ with  the open interval $(-0.5,0.5)$ endowed
    with periodic boundary conditions.
    Let $K\in C^1(\mathbb T)$  be even and such that $\partial_xK<0$ on $(0,0.5)$.
    Consider the system
\begin{equation}\label{Eq:Ergodic1D}\tag{$\mathbf{MFG}_{\text{ergo,1D}}$}
\begin{cases}
    \overline \lambda-\frac{(\overline u')^2}2=\overline\theta\,, 
\\  (\overline mu')'=0\,, 
\\ -\mu\overline\theta''=\overline\theta(K-\e \overline m-\overline\theta)\,, 
\\ \overline\theta\geq 0,\; \overline\theta\not\equiv 0,\;
    \int_{\mathbb{T}}\overline u=0\,, \overline m\geq 0\,, \int_{\mathbb T} \overline m=1.
\end{cases}
\end{equation}
    If $\e\in (0,\int_{\mathbb{T}} K)$,
    there exists a unique solution $(\overline \lambda,\overline\theta,\overline u,\overline m)$ of \eqref{Eq:Ergodic1D}
    with $\overline\theta\neq 0$.
    If $\e>\int_{\mathbb{T}} K$, any solution of \eqref{Eq:Ergodic1D} satisfies $\overline\theta=0$ a.e. in $\mathbb{T}$.
\end{theorem}
One should note that under the assumptions of this last theorem, the equation defining $\overline \theta$
    is well-posed for any probability measure $m\in \mathcal P(\mathbb{T})$
such that $\lambda_1(K-m,\mu)<0$.
    This is due to the Sobolev embedding $W^{1,2}(\mathbb{T})\hookrightarrow \mathscr C^0(\mathbb{T})$
    and the characterisation
    of $\overline{\theta}$ as the (unique) solution of 
\[ \min_{\theta\in W^{1,2}(\T)\,,\theta\geq 0}\frac{\mu}2\int_\T (\partial_x\theta)^2-\frac12\int_\T (K-\e m)\theta^2+\frac13\int_\T \theta^3.\]
    To interpret Theorem \ref{Th:Ergodic1D}, it is important to note that the condition
    $\int_T \Kt<0$ (where $\Kt$ stands for $K-\e m$ here)
    is not in itself sufficient to guarantee that $\overline\theta_{\Kt}=0$,
    where $\overline\theta_{\Kt}$ solves \eqref{Eq:ELD};
    we refer to classical results \cite{BHR}.
    Consequently, the last sentence of Theorem \ref{Th:Ergodic1D} can be rewritten as:
    when the fishing pressure exceeds a fixed threshold,
 	the competition between fishermen  leads to an extinction of the fishes,
    and each agent, in this situation, has a zero outcome.
    However, even above the aforementioned threshold,
    it is easy to construct strategies (that are not individually optimal) where each fisherman obtains a positive outcome, and the total fishes' population survives in the long run.

    Those considerations make Theorem \ref{Th:Ergodic1D}  
    an instance of the \emph{tragedy of the commons} in the following sense:    greedy individual strategies may lead to the extinction of
    the species,
    while coordination between fishermen may ensure survival of the species and better outcomes for every fisherman involved.

\subsection{Long-time behaviour of \eqref{Eq:MFGCompet} and convergence to \eqref{Eq:ErgodicMFG}}\label{Se:KAM}
Let us now pass to the study of 
the long-time behaviour of solutions to \eqref{Eq:MFGCompet}.
\begin{theorem}\label{Th:ErgodicMFG2nd}
    Assume \ref{H:K}, \ref{H:Theta0} hold and that $\e\in(0,\e_0)$,
    where $\e_0$ is given in Theorem \ref{Th:UniquenessMFGCompet}.
    Take 
    $(\overline \lambda,\overline\theta_{\e,\overline m},\overline m,\overline V)$
    and $(\theta_{\e,m}^T,m^T\,, u^T)$ as the unique solutions of \eqref{Eq:ErgodicMFG}
    and \eqref{Eq:MFGCompet} respectively, for $T>0$.
    Define the following rescaled functions:
    for $(s,x)\in [0,1]\times \T$,
    \begin{equation}
        \label{eq:def_rescaled}
        \Theta^T(s,x):=\theta^T(sT,x),\;
        M^T(s,x):=m(sT,x),\;
        W^T(s,x):=u^T(sT,x)
    \end{equation}
Then there holds:
\begin{enumerate}
\item $\Theta^T\underset{T\to\infty}\rightarrow \overline \theta_{\e,\overline m}$ strongly in $L^2(0,1;L^2(\T))$ and, in fact:
\[\Vert\Theta^T-\overline\theta_{\e,\overline m}\Vert_{L^2((0,1)\times\T)}^2\leq\frac{C}{T}\]for some constant $C$.
\item $W^T/\overline\lambda \underset{T \to \infty}\rightarrow \overline u$ strongly in $L^\infty(0,1;L^\infty(\T))$ and, in fact
\[ \Vert \frac{W^T}T-\overline\lambda(1-s)\Vert_{L^\infty(\TT)}\leq\frac{C}{\sqrt{T}}.\]
\end{enumerate}
\end{theorem}

\subsection{Bibliography and related works}\label{Se:Biblio}
\paragraph{Mean Field Games}
MFG were introduced independently by Lasry \& Lions \cite{LasryLions1,LasryLions2,LasryLions3} and by Caines, Huang \& Malham\'e \cite{Malhame2,Malhame1} in order to model and analyse game theoretical situations in the presence of a large number of agents. The bulk of the theory was developed by Lions in his series of lectures at Coll\`ege de France; we refer to the lecture notes of Cardaliaguet \cite{CardaliaguetCdF}, as well as to the self-contained lectures of Cardaliaguet \& Porretta
\cite{Cardaliaguet_2020}. The interested reader might profit from consulting the volume \cite{Achdou_2020}.  A core difficulty in the study of MFG is the identification of regimes of uniqueness for the systems. As we briefly sketched in the introduction, there are essentially two regimes for which it is well-known that uniqueness holds. The first one is the monotonous setting, while the second one is the small interaction case. Lasry \& Lions \cite[pp. 683]{LasryLions2} already observed, by drawing a connection with the Euler equation in the one-dimensional case, that uniqueness could fail. Bardi \& Fischer \cite{Bardi_2019} constructed a variety of explicit counter-examples to uniqueness by picking suitable Hamiltonians and ensuring that the coupling between the HJB equation and the FPP equation was ``non-monotonous" enough. We also refer to the recent \cite{zbMATH07729952} by Alpar \& M\'esz\'aros for a comprehensive discussion of (non-)uniqueness issues. Achdou \& Kobeissi \cite{achdou2021mean} computed multiple numerical solutions of a discretised MFG system,
conjecturing that non-uniqueness should hold when the discretisation step goes to zero.
In our case, we were not able to obtain a non-uniqueness regime. We leave this as an interesting and challenging open problem. 

\paragraph{Long-time behaviour of Hamilton-Jacobi equations, time asymptotics for MFG systems and comparison with \cite{CardaliaguetKAM}}
One of the core results of the present paper, the long-time behaviour of \eqref{Eq:MFGCompet}, fits into a general line of research on the long-time behaviour of viscosity solutions of Hamilton-Jacobi equations in general, and of solutions of MFG systems in particular. The time asymptotics of Hamilton-Jacobi equations have been the subject of an intense research activity since the 80's. Let us mention the foundational paper of Kru\v zkov \cite{Kruzkov} for quasilinear equations, and the contribution of Dafermos \cite{Dafermos} for one-dimensional conservation laws. For general Hamilton-Jacobi problems, we refer in particular to the first results obtained by Lions in \cite[Section 14.2]{LionsGeneralized}; this provided the impetus for the works of (among others) Lions, Papanicolaou \& Varadhan \cite{LionsPapanicolaouVaradhan}, Barles \cite{Barles1985} and the class of general results that were later obtained  by Fathi \cite{Fathi_1998} and Namah \& Roquejoffre \cite{Namah_1999}. The results of Fathi pointed to a deep connection with the theory of dynamical systems, particularly through the weak KAM theory already developed by Fathi in \cite{FATHI19971043}. In \cite{Barles_2000}, these results were generalised by Barles \& Souganidis using different methods. Regarding the weak KAM theory, we refer to \cite{E1999AUBRYMATHERTA,Evans_2008,FATHI19971043}.  As for the time asymptotics of MFG systems, to the best of our knowledge, the first contribution was the work of Cardaliaguet, Lasry, Lions \& Porretta \cite{CardaliaguetLasryLionsPorretta} in the case of second-order MFG (that is, for stochastic control problems). In the second order setting, the diffusive part of the system makes it so that the Lasry-Lions condition is sufficient to obtain long time convergence of the MFG system to the associated ergodic system. A quantified Lasry-Lions monotonicity conditions provides stronger (and explicit) convergence rates, and no quadraticity of the Hamiltonian is required to obtain the result. In the first order case, which is more intimately linked to the questions under consideration here, the paper of Cardaliaguet \cite{CardaliaguetKAM} is the first to consider the long-time behaviour under two key assumptions. The first one is that the Hamiltonian is essentially quadratic (\emph{i.e.} its Hessian is uniformly bounded from below and from above by two fixed positive definite matrices). The second one, which is more problematic, is that \cite{CardaliaguetKAM} requires a quantified monotonicity condition. To be more precise, the system under consideration in \cite{CardaliaguetKAM} writes
\[ \begin{cases}-\partial_t u-H(\n_x u)=F[m]\,, u(T)=0,
\\ \partial_t m+\nabla\cdot(\n_pH(\n_x u)m)=0\,, m(0)=m_0,\end{cases}\]
The key assumption on the coupling term $F:\mathcal P(\T)\to \mathscr C^2(\T)$ is that there exists a constant $C>0$ such that, for any $(m_1,m_2)\in \mathcal P(\T)$, 
\begin{equation}\label{Eq:HypoCarda} \iint_{(0,T)\times \T}\left(F[m_1]-F[m_2]\right)d(m_1-m_2) \leq-C \iint_{(0,T)\times \T}\left(F[m_1]-F[m_2]\right)^2.\end{equation}
A particularly important class of couplings satisfying this assumption can be found in \cite[Example 1.1]{CardaliaguetKAM} and consists of functions $F$ writing
$F(m)=(m\star \xi)\star \xi$ for smooth, compactly supported kernels $\xi$. In particular, this covers non-local in space couplings. It can also cover non-local in time couplings.
In the present paper, the main difficulty to get the long-time behaviour
is that \eqref{Eq:HypoCarda} is quite difficult to obtain, and deriving it is the bulk of our paper.
To be more precise, consider the case of a linearly controlled population, and of a similar fishing problem, which would write
\begin{equation}
\begin{cases}
-\partial_t u-\frac{\Vert \n_xu\Vert^2}2=\rho\star\Psi_m=:F[m]&\text{ in }(0,T)\times \T\,,
\\ u(T,\cdot)\equiv 0,
\\\partial_tm+\nabla \cdot\left(m\n_x u\right)=0&\text{ in }(0,T)\times \T\,, \\m(0,\cdot)=m_0,
\\ \partial_t \Psi_m-\mu \Delta\Psi_m=K-\e(\rho\star m)&\text{ in }(0,T)\times \T\,, 
\\ \Psi_m(0,\cdot)=0.
\end{cases}
\end{equation}
Observe two things: first, for any two probability measures $(m_1,m_2)$ there holds, for any $t\in [0,T]$, $\int_\T \Psi_{m_1}=\int_\T \Psi_{m_2}$. Second, $z:=\Psi_{m_1}-\Psi_{m_2}$ solves
$\partial_t z-\mu \Delta z=-\e\rho\star(m_1-m_2)$ so that 
\begin{multline*} \iint_{(0,T)\times \T}(F[m_1]-F[m_2])d(m_1-m_2)=\iint_{(0,T)\times \T} z (\rho\star(m_1-m_2))\\=-\frac1\e\left(\int_\T z(T,\cdot)^2+\mu\iint_{(0,T)\times \T}|\n z|^2\right).\end{multline*} By the Poincar\'e-Wirtinger inequality we deduce that 
\[ \iint_{(0,T)\times \T} (F[m_1]-F[m_2])d(m_1-m_2)\leq-\frac{\mu}{\e}\iint_{(0,T)\times \T}z^2\lesssim-\frac{\mu}\e\iint_{(0,T)\times \T}\left(F[m_1]-F[m_2]\right)^2.\] However, such direct computations are not amenable to non-linearities, and the technique which we use to obtain Proposition \ref{Pr:EpsilonMonotone} is much more involved.

\paragraph{Spatial ecology and optimal control}
Another line of research this contribution fits into is the study of the influence of spatial heterogeneity on population dynamics. Moving away from the spatially homogeneous models of Skellam \cite{Skellam}, many researchers investigated the influence of spatial heterogeneity on mathematical ecology. As an important reference, let us point to the monograph \cite{ShigesadaKawaski}. In the wake of the works of Cantrell \& Cosner \cite{CantrellCosner1} and of Berestycki, Hamel \& Roques \cite{BHR},  a popular point of view on the topic has been that of optimising the spatial heterogeneity. This led to several important articles devoted to the understanding of spectral optimisation problem and, following Lou \cite{LouInfluence}, focusing on non-energetic optimisation problem such as biomass optimisation. Although very different in essence, as we are here dealing with agents trying to harvest optimally, rather than trying to optimise a domain for the fishes' population, we mention this line of research as the method we use to prove Proposition \ref{Pr:EpsilonMonotone} actually derives from considerations on the optimisation of the total biomass \cite{heo2021ratio,LiangZhang,Mazari2021,SuTongYang}, as we will explain in the course of the proof. 

While all of the aforementioned contributions deal with ``benevolent'' agents,
there has been an intense activity in recent years concerning ``nefarious'' agents,
wherein the control is optimised so as to eradicate an invasive species in the most efficient way. We refer, for instance, to \cite{Almeida_2022,almeida:hal-03811940,Bressan_2022,Duprez_2021}.

\paragraph{Game theoretical problems and spatial ecology}
The literature devoted to game theoretical problems in spatial ecology is small, but growing rapidly. To briefly describe these contributions, let us start by problems related to Nash equilibria models. Bressan,  Coclite, Shen \& Staicu \cite{Bressan_2013,bressan2010measure,bressan2019competitive} investigated optimal harvesting problems, but within the framework of elliptic equations with measure data. A key aspect of their work is that the measure to be optimised is the density of players. What they prove (among other things) is that there exist measures satisfying optimality conditions for the game theoretical problem under consideration. In a previous article by the last two authors \cite{Mazari_2022}, Nash equilibria, their existence and qualitative properties, were investigated in situations where a finite number of players is involved. Let us conclude by mentioning our previous paper \cite{KMFRB23} in which we investigated a MFG system for  harvesting, but with two major differences. The first one is that \cite{KMFRB23} focuses on bistable models for the fishes' population, while the second one has to do with the fact that \cite{KMFRB23} focuses on special solutions in the whole of $\R^d$, namely, traveling waves, which are known to be important in the dynamical properties of the system. The methods of \cite{KMFRB23} are essentially constructive, and are completely different from the ones presented here.

\section{Proof of Theorems \ref{Th:ExistenceMFGCompet}: existence and regularity of solutions}
\label{sec:exist}
\subsection{The time dependent system}
\label{subsec:exist_TD}
As is customary in MFG, for the time dependent system \eqref{Eq:MFGCompet},
the existence of solutions relies on a fixed point argument
that requires some compactness properties given by a priori estimates.
Here, those estimates are stated in the subsequent proposition.
\begin{proposition}\label{Pr:UniformRegularityEstimates}
There exists a constant $C=C(\e,\delta)$ such that, for any $\nu>0$, for any $T>0$, if $(u,m,\theta_{\e,m})$ is a solution of \eqref{Eq:MFGCompet}, there holds:
\begin{align}
    \label{Eq:Goal}
    &\Vert \n_x u\Vert_{L^\infty((0,T)\times \T)}
    +\Vert \partial_t u\Vert_{L^\infty((0,T)\times \T)}\leq C.
    \\
    \label{Eq:Goal3}
    &\Vert \theta_{\e,m}-\overline \theta\Vert_{L^\infty(0,T; \mathscr C^2(\T))}
    +\Vert \partial_t \theta_{\e,m}\Vert_{L^\infty((0,T)\times \T)}
    \leq C \e \Vert \rho\Vert_{\mathscr C^2}.
\end{align}
\end{proposition}
Observe that the latter estimates are uniform in time;
while this is not necessary \emph{per se} when discussing the existence of solutions in a bounded time interval,
this is important for the study of the long-time behaviour. Since the proof is quite lengthy, we first show how it implies Theorem \ref{Th:ExistenceMFGCompet}.

\begin{proof}[Proof of Theorem \ref{Th:ExistenceMFGCompet} for \eqref{Eq:MFGCompet}] 
For a fixed $\nu>0$, the existence of solutions to \eqref{Eq:MFGCompet} follows from Proposition \ref{Pr:UniformRegularityEstimates}
and the Leray-Schauder fixed point theorem.
Regarding the first-order case $\nu=0$, take a decreasing sequence $\{\nu_k\}_{k\in \N}$ of viscosities vanishing to 0 and, for any $k\in \N$, let 
$(u_{k},m_{k},\theta_{k})$ be a solution of \eqref{Eq:MFGCompet} with viscosity $\nu_k$.
Observe that the estimates in Proposition \ref{Pr:UniformRegularityEstimates}
are uniform in $k\in \N$.
Therefore, up to extracting a subsequence, 
$\{\theta_k\}_{k\in \N}$  converges in $C^{0,1}([0,T]\times\T)$ to some limit $\theta$, $\{u_k\}_{k\in \N}$ converges to a limit $u$ that is the unique viscosity solution of the HJB equation associated with $\theta$ and $\{m_k\}_{k\in \N}$ converges in the distributional sense to the solution of the Fokker-Planck equation. Finally, standard continuity estimates on Fokker-Planck equations \cite[Lemma 5.3]{CardaliaguetKAM} provide the continuity of $m$ for the Wasserstein distance, thereby providing the existence of a solution to \eqref{Eq:MFGCompet}.
\end{proof}
It remains to prove Proposition \ref{Pr:UniformRegularityEstimates}. We begin with an auxiliary result.

\begin{proposition}\label{Pr:RegularityNonLocal} Let $K$ satisfy \eqref{H:Theta0} and let $\overline \theta_K$ be the solution of \eqref{Eq:ELD}.
Let $m\in L^\infty(0,\infty;\mathcal M_+(\O))$ and, for any $\e\geq 0$, let $\theta_{\e,m}$ be the unique solution of 
\begin{equation}\begin{cases}
\partial_t\theta_{\e,m}-\mu \Delta \theta_{\e,m}-\theta_{\e,m}(K-\theta_{\e,m})-\e (\rho\star m)\theta_{\e,m}=0&\text{ in }[0,+\infty)\times \T\,, 
\\ \theta_{\e,m}=\theta_0\end{cases}\end{equation} There exists a constant $C>0$ such that 
\[ \Vert \theta_{\e,m}-\overline \theta_K\Vert_{L^\infty(0,+\infty; \mathscr C^2(\T))}+\Vert \partial_t \theta_{\e,m}\Vert_{L^\infty((0,+\infty)\times \T)}\leq C \e \Vert \rho\Vert_{\mathscr C^2}\Vert m\Vert_{L^\infty(0,+\infty;\mathcal M(\T))}.
\]
\end{proposition}

\begin{proof}[Proof of Proposition \ref{Pr:RegularityNonLocal}]
We begin with a weaker estimate, namely, that there exists a constant $C$ such that
\begin{equation}\label{Eq:Est0}\Vert \theta_{\e,m}-\overline \theta_K\Vert_{L^\infty(0,+\infty; \mathscr C^0(\T))}\leq C \e \Vert \rho\Vert_{L^\infty}\Vert m\Vert_{L^\infty(0,+\infty;\mathcal M(\T))}.
\end{equation}
First of all, observe that we have 
\[ 0\leq \e (\rho\star m)\leq \e  \Vert \rho\Vert_{L^\infty}\Vert m\Vert_{L^\infty(0,+\infty;\mathcal M(\T))}=:\e M_1.\] 

Let $\overline\theta_{K-\e M_1}$ be the unique solution of \eqref{Eq:ELD} with $K$ replaced with $K-\e M_1$. For $\e$ small enough, $\overline\theta_{K-\e M_1}$ is uniquely defined and satisfies $\inf_\T \overline\theta_{K-\e M_1}>0$.  Now, observe that $z=\overline\theta-\overline\theta_{K-\e M_1}$ satisfies
\begin{equation}\label{eq:z} -\mu \Delta z-z(K-(\overline\theta_{K-\e M_1}+\overline \theta_K))=\e M_1 \overline\theta_{K-\e M_1}\geq 0. \end{equation}
Observe that $\lambda_1(\mu,K-\overline\theta)=0$: indeed, $\overline\theta$ is a positive eigenfunction associated with the eigenvalue zero. Since $\overline\theta$ has constant sign, it is associated with the lowest eigenvalue. As $ \overline\theta_{K-\e M_1}>0$, it is an easy consequence of the variational formulation \eqref{Eq:Rayleigh} that
\begin{equation}\label{Eq:EvIn}\lambda_1(\mu,K-(\overline\theta+ \overline\theta_{K-\e M_1}))>0\end{equation} whence we first deduce that
\[ z\geq 0\text{ in }\T.\] Indeed, should the negative part $z_-$ of $z$ be non-zero, multiplying \eqref{eq:z} by $z_-$ and integrating by parts leads to 
\[ \mu\int_\O |\n z_-|^2-\int_\T z_-^2 (K-(\overline\theta+ \overline\theta_{K-\e M_1}))\leq 0,\] in contradiction with \eqref{Eq:EvIn}

Furthermore, multiplying \eqref{eq:z} by $z$ and integrating by parts we obtain 
\[ \lambda_1(\mu,K-(\overline\theta+ \overline\theta_{K-\e M_1}))\int_\O z^2\leq \e M_1 \left(\Vert \overline\theta_K\Vert_{L^\infty}+\Vert \theta_{\e,m}\Vert_{L^\infty(\T)}\right)\Vert z_1\Vert_{L^2(\O)}.\] Now observe that we have 
\[-\mu\Delta z=f:=\e M_1\overline \theta_{K-\e M_1}+z\left(K-(\overline\theta_{K-\e M_1}+\overline \theta_K)\right).\] We thus deduce that 
\[ \Vert z\Vert_{W^{1,2}(\O)}\lesssim \Vert f\Vert_{L^2(\T)}\] 
and a simple bootstrap argument leads to the following conclusion: there exists a constant $C_0$ such that 
\[ \Vert z\Vert_{L^\infty(\T)}\leq C_0\e M_1.\] Overall, we have thus established that 
\begin{equation}\label{Eq:Est1} 0\leq \overline \theta_K-\overline\theta_{K-\e M_1}\leq C_0\e M_1.\end{equation}

Coming back to $\theta_{\e,m}$, observe that the maximum principle implies that $\theta_{\e,m}\geq 0$ in $[0,+\infty)\times \R^d$. In particular, $\theta_{\e,m}$ satisfies
\[ \theta_{\e,m}\leq \overline\theta.\] Furthermore, using \eqref{Eq:Est1}, it is readily checked that $\theta_{\e,m}$ is a super-solution of the equation
\[\begin{cases}\partial_t\psi_\e-\mu\Delta \psi_\e-\psi_\e(K-\e M_1-\psi)=0&\text{ in }[0,+\infty)\times \T\,, 
\\ \psi_t(0,\cdot)=\overline\theta_{K-\e M_1}\end{cases}\] whose solution is $\psi=\overline\theta_{K-\e M_1}$.  This implies $\overline\theta_{K-\e M_1}\leq \theta_{\e,m}$ which, in combination with \eqref{Eq:Est1}, yields 
\[ 0\leq \overline\theta_K-\theta_{\e,m}\leq C \e M_1,\] thereby establishing \eqref{Eq:Est0}.

We now proceed to obtain a uniform control of the $\mathscr C^1$ norm of $\theta_{\e,m}$. Let $i\in \{1,\dots,d\}$ and let $\p_{\e,m}:=\partial_{x_i} \theta_{\e,m}$. Clearly, $\p_{\e,m}$ satisfies
\begin{equation}\label{Eq:EquationDerivative}
\partial_t\p_{\e,m}-\mu \Delta \p_{\e,m}-\p_{\e,m}\left(K-\e (\rho\star m)-2\theta_{\e,m}\right)=-\e \theta_{\e,m}(\partial_{x_i}\rho)\star m+(\partial_{x_i}K)\theta_{\e,m}.\end{equation}
From the same argument that led to \eqref{Eq:EvIn}, \eqref{Eq:Est0} yields the existence of $\underline\lambda>0$ such that, for any $t\geq 0$ and any $\e>0$ small enough, there holds
\begin{equation}\label{Eq:EstLambda} 0<\underline\lambda\leq \lambda(\mu,K-\e(\rho\star m)-2\theta_{\e,m}).\end{equation} Indeed, observe that by \eqref{Eq:Rayleigh}, there exists a constant $c$ such that \[\left| \lambda(\mu,K-\e(\rho\star m)-2\theta_{\e,m})-\lambda(\mu,K-2\overline\theta_K)\right|\leq c\e.\]

We claim that \eqref{Eq:EstLambda} suffices to obtain a first uniform $L^\infty$ estimate on $\p_{\e,m}$. For the sake of readability, we let 
\[V_\e:=(K-\e(\rho\star m)-2\theta_{\e,m})\text{ and } f_\e:=-\e(\partial_{x_i}\rho)\star m+(\partial_{x_i}K)\theta_{\e,m}.\] Observe in particular that 
\begin{equation}\label{Eq:EstfEps} \Vert f_\e\Vert_{L^\infty((0,+\infty)\times \T)}\leq C_1(\e\Vert \rho\Vert_{\mathscr C^1(\T)}+\Vert K\Vert_{\mathscr C^1(\T)}).\end{equation}  With these notations, $\p_{\e,m}$ solves
\[ \partial_t \p_{\e,m}-\mu \Delta \p_{\e,m}-V_\e \p_{\e,m}=f_\e.\]
We need two simple lemmas to allow us to conclude.
\begin{lemma}\label{Le:Regularity}
Let $V_\e$ be given as above. Then, for any $g\in L^\infty(\T)$, the solution $\p$ of 
\begin{equation}\label{Eq:g}\begin{cases}
\partial_t \p-\mu \Delta \p-V_\e \p=0&\text{ in }[0,+\infty)\times \T\,, 
\\ \p(0,\cdot)=g\end{cases}\end{equation} satisfies 
\[ \Vert \p\Vert_{L^\infty((0,+\infty)\times \T)}\leq C_2e^{-\overline\lambda t} \Vert g\Vert_{L^\infty(\T)}\] for some constants $C_2$ and $\overline\lambda$ that only depend on the dimension,  $\Vert m\Vert_{L^\infty(0,T;\mathcal M(\T))}$,  $\delta$ and  $K$.
\end{lemma}
Observe that in one dimension and for a constant $g$, this is merely a consequence of the standard parabolic regularity estimate
\[ \Vert \p\Vert_{L^\infty(0,\infty;W^{1,2}(\mathbb T))}\leq \Vert g\Vert_{L^\infty(\mathbb T)}e^{-\int_0^t \lambda_1(\mu,V_\e(s))ds}\] and of the Sobolev embedding $W^{1,2}\hookrightarrow \mathscr C^0$.
\begin{proof}[Proof of Lemma \ref{Le:Regularity}]
Observe that by the maximum principle and by linearity it suffices to treat the case $g\equiv 1$ in $\T$. The maximum principle yields, in this case, 
\[ \p\geq 0,\] so that
\[ V_\e\p=(K-\e(\rho\star m)-2\theta_{\e,m})\p\leq (K-2\theta_{\e,m})\p.\] Defining $\overline V:=K-2\overline \theta_K$ we obtain
\begin{equation}\label{Eq:Annex} \partial_t \p-\mu\Delta \p-\overline V\p\leq 2(\overline \theta_K-\theta_{\e,m})\p.\end{equation} Since $\p\geq 0$, \eqref{Eq:Est0} gives
\begin{equation}\label{Eq:LeSub}\partial_t \p-\mu \Delta \p-(\overline V-2\e M_1)\p\leq 0.\end{equation} There exists $\overline\lambda>0$ such that, for $\e>0$ small enough, 
\[0< \overline \lambda\leq \lambda_1(\mu,\overline V-2\e M_1).\]
Let $\overline\p_\e$ be the ($L^2$ normalised) eigenfunction associated with $\lambda_1(\mu,\overline V-2\e)$. By the strong maximum principle and elliptic regularity, there exists $\underline d>0$ such that, for any $\e>0$ small enough, $\underline d\leq\inf_{ \O} \p_\e\leq \sup_\T \p_\e\leq \frac1{\underline d}$.  We  deduce from \eqref{Eq:LeSub} that 
\[ 0\leq \p\leq \frac{e^{-\overline\lambda t}}{\underline d}\Vert\p_\e\Vert_{L^\infty(\T)}\leq \frac1{\underline d^2}e^{-\overline \lambda t}.\]

\end{proof}The next result we need is a Duhamel-type formula.
\begin{lemma}\label{Le:Duhamel}
Define, for any $s\geq 0$, the solution $\p_s$ of 

\begin{equation}\begin{cases}
\partial_t\p_s-\mu\Delta \p_s-V_\e\p_s=0&\text{ in }[s,+\infty)\times \T\,, 
\\ \p_s(s,\cdot)=f_\e(s).\end{cases}
\end{equation}
Then 
\[ \p_{\e,m}(t,\cdot)=\int_0^t \p_s(t,\cdot)ds.\]
\end{lemma}
The proof is a straightforward computation. 

Lemmata \ref{Le:Regularity}-\ref{Le:Duhamel} allow to go back to the proof of Proposition \ref{Pr:RegularityNonLocal} as follows: from Lemma \ref{Le:Duhamel} we have 
\[ \Vert \p_{\e,m}(t,\cdot)\Vert_{L^\infty(\T)} \leq \int_0^t \Vert \p_s(t,\cdot)\Vert_{L^\infty(\T)}ds.\] Lemma \ref{Le:Regularity} in turn yields
\[ \Vert \p_{\e,m}(t,\cdot)\Vert_{L^\infty(\T)}\leq \int_0^t e^{-\overline\lambda(t-s)}\Vert f_\e(s)\Vert_{L^\infty(\T)}ds\leq C_3
\]
for some constant $C_3$, where we used \eqref{Eq:EstfEps}.

To obtain the $\mathscr C^1$ control of $\theta_{\e,m}$ we let $z:=\partial_{x_i}\theta_{\e,m}-\partial_{x_i}\overline\theta$. From \eqref{Eq:EquationDerivative}, $z$ solves
\begin{multline*} \partial_t z-\mu \Delta z-z(K-2\overline \theta_K)=-\e(\partial_{x_i}\rho\star m)\theta_{\e,m}+(\partial_{x_i}K)(\theta_{\e,m}-\overline\theta)-2\p_{\e,m}(\overline\theta-\theta_{\e,m}).\end{multline*} We deduce from the previous estimates that there exists a constant $C_4$ such that 
\[-C_4\e \leq \partial_t z-\mu \Delta z-z(K-2\overline \theta_K)\leq C_4\e.\] From the same arguments as before, we conclude that 
\[\Vert z\Vert_{L^\infty((0,+\infty)\times \T)}\leq C_5\e,\] whence the conclusion:
\begin{equation}\label{Eq:Est2}
\Vert \overline\theta-\theta_{\e,m}\Vert_{L^\infty(0,+\infty;\mathscr C^1(\T))}\leq C_5 \e.\end{equation}

\paragraph{Control of the $\mathscr C^2$ norm} Our goal is now to prove
\begin{equation}\label{Eq:Est3}
\Vert \overline\theta-\theta_{\e,m}\Vert_{L^\infty(0,+\infty;\mathscr C^2(\T))}\leq C \e\end{equation} for some constant $C$. To do so, we use the approach introduced to derive \eqref{Eq:Est2}, and we let $\p_{\e,m}:=~\partial^2_{x_ix_k}\theta_{\e,m}$, for $i\,, j\in \{1,\dots,d\}$. Then, $\p_{\e,m}$ satisfies
\begin{multline}
\label{Eq:SecondDerivative}
\partial_t \p_{\e,m}-\mu \Delta \p_{\e,m}-V_\e\p_{\e,m}=\theta_{\e,m}\partial^2_{x_ix_j}K-\e(\partial^2_{x_ix_j}\rho\star m)\theta_{\e,m}+\partial_{x_i}K\partial_{x_j}\theta_{\e,m}+\partial_{x_j}K\partial_{x_i}\theta_{\e,m}\\-\e\left(\partial_{x_i}\theta_{\e,m}(\partial_{x_j}\rho\star m)+\partial_{x_j}\theta_{\e,m}(\partial_{x_i}\rho\star m)\right)\end{multline}
Using once again Lemmata \ref{Le:Regularity}-\ref{Le:Duhamel} we deduce that, uniformly in $\e$, we have 
\begin{equation}\label{Eq:Est4}
\Vert \p_{\e,m}\Vert_{L^\infty(0,+\infty;\mathscr C^2(\T))}\leq C \end{equation} for some $C$. Observing that $\overline\p:=\partial_{x_ix_j}\overline\theta$ solves 
\[ \partial_t \overline\p-\mu \Delta \overline\p-\overline V\overline\p=\overline\theta\partial^2_{x_ix_j}K+\partial_{x_j}K\partial_{x_i}\overline\theta+\partial_{x_i}K\partial_{x_j}\overline\theta\] and taking the difference $z:=\overline \p-\p_{\e,m}$, we obtain \eqref{Eq:Est3} exactly in the same way that we derived \eqref{Eq:Est2}.  Consequently, we deduce that 
\begin{equation}\label{Eq:Est5}
\Vert \partial_t \theta_{\e,m}\Vert_{L^\infty((0,+\infty)\times\T)}\leq C \e,\end{equation} thereby concluding the proof.
\end{proof}

In the following result we work on the solution of the HJB equation
\[ -\partial_t u-\nu\Delta u-\frac{\Vert \n_xu\Vert^2}2=\rho\star \theta_{\e,m}\] and, since the right-hand side of the equation is uniformly bounded in $\mathscr C^2$, we actually aim at deriving a regularity result for the solution of 
\begin{equation}\label{Eq:HJB2}
\begin{cases}
-\partial_t u_T-\nu\Delta u_T-\frac{\Vert \n_x u_T\Vert^2}2=F&\text{ in }(0,T)\times \T\,, 
\\ u_T(T,\cdot)=0.\end{cases}\end{equation}
We now prove the following regularity bound:
\begin{proposition}\label{Pr:RegularityHJB} Let $F\in L^\infty(0,T;\mathscr C^2(\T))$ and $u_T$ be the unique solution of \eqref{Eq:HJB2}.
 There exists a constant $C$ such that for any $T\geq 0$  and any $\nu>0$ there holds
\begin{equation}\label{Eq:Goal} \Vert \n_x u\Vert_{L^\infty((0,T)\times \T)}+\Vert \partial_t u\Vert_{L^\infty((0,T)\times \T)}\leq C(1+\nu)\Vert F\Vert_{L^\infty(0,T;\mathscr C^2(\T))}.\end{equation}  In fact, $V$ is uniformly semi-convex: there exists a constant $C'$ such that for any $\nu>0$ there holds
\begin{equation}\label{Eq:SConcave}
 \n^2_xu\geq- C'\Vert F\Vert_{L^\infty(0,T;\mathscr C^2(\T))}\mathrm{I}_d
\end{equation} where we used the standard order on the set of symmetric matrices.
\end{proposition}
The general strategy of proof of this estimate follows the usual semi-concavity method, see for instance \cite{Bardi_2023}.
\begin{proof}[Proof of Proposition \ref{Pr:RegularityHJB}]
We will actually show that \eqref{Eq:Goal} holds for any $\nu>0$, with a constant $C$ independent of $\nu$. Since $T$ is considered fixed, we drop the subscript $T$ in $u_T$. Since $u$ is a viscosity solution of \eqref{Eq:HJB2} when $\nu=0$, it then suffices to take the limit $\nu \to 0$ to obtain the required result. Thus, a positive $\nu>0$ is given. Picking a direction $\xi\in \mathbb S^{d-1}$ and letting $W:=\frac{\partial^2 u}{\partial \xi^2}$,  the fact that $W^2\leq \Vert \n_\xi \n_x u\Vert^2$ implies that $W$ is a solution of 
\[ \begin{cases}
-\partial_t W-\nu \Delta W-\langle \n_xu,\n_xW\rangle- W^2\geq -\Vert F\Vert_{L^\infty(0,T;\mathscr C^2(\T))}
\\ W(T,\cdot)=0.
\end{cases}
\] The maximum principle implies
\[ W\geq - \Vert F\Vert_{L^\infty(0,T;\mathscr C^2(\T))}.\] As a consequence, $u$ is uniformly semi-convex in $x$. Since $\T$ is a bounded domain, standard results about semi-convex functions (see for instance \cite{Alberti_1992}) imply that $V$ is uniformly (in $x$, in $T$ and in $t\in[0,T]$) Lipschitz-continuous with respect to the $x$ variable. Thus, there exists a constant $C$ such that 
\begin{equation}\label{Eq:HJBInterm}
\Vert \n^2_xu\Vert_{L^\infty((0,T)\times \T)}+\Vert \n_xu\Vert_{L^\infty((0,T)\times \T)}\leq C\Vert F\Vert_{L^\infty(0,T;\mathscr C^{2}(\T))}.\end{equation}Plugging \eqref{Eq:HJBInterm} in \eqref{Eq:HJB2} we conclude that
\[ \Vert \partial_t u\Vert_{L^\infty((0,T)\times \T)}\leq C(1+\nu)\Vert F\Vert_{L^\infty(0,T;\mathscr C^{2}(\T))}.\]

\end{proof}

We can finally prove Proposition \ref{Pr:UniformRegularityEstimates}.

\begin{proof}[Proof of Proposition \ref{Pr:UniformRegularityEstimates}]
 First of all, integrating the equation in $m$ in space yields
 \[ \forall t\in [0,T]\,, \int_\T m(t,\cdot)=\int_\T m_0.\] Proposition \ref{Pr:RegularityNonLocal} gives the required estimates on $\theta_{\e,m}$. Plugging the estimates in $\theta_{\e,m}$ in the Hamilton-Jacobi-Bellman equation and using Proposition \ref{Pr:RegularityHJB} we obtain the desired bounds on $V$.
 
\end{proof}

\subsection{The ergodic system}
\label{subsec:exist_KAM}
The existence of a solution to the ergodic MFG system \eqref{Eq:ErgodicMFG}
follows from a direct adaptation of the arguments of \cite{CardaliaguetLasryLionsPorretta} in the second-order case, and we will not repeat them here. We only consider the first-order setting  whose proof relies on the weak KAM formula from \cite{CardaliaguetKAM} , itself based on \cite{Fathi_1998,FATHI19971043}.

\begin{proof}[Proof of Theorem \ref{Th:ExistenceMFGCompet} for \eqref{Eq:ErgodicMFG} with $\nu=0$.]
Consider, for a fixed $m$, the Lagrangian 
\[ L_m(x,v):=\frac{\Vert v\Vert^2}2-(\rho\star \overline\theta_{\e,m})\text{ where }\begin{cases}-\mu\Delta\overline\theta_{\e,m}-\overline\theta_{\e,m}(K-\e(\rho\star m)-\overline\theta_{\e,m})=0&\text{ in }\T\,, 
\\ \overline\theta_{\e,m}\geq 0\,, \neq 0.\end{cases}\]
Let $E_m$ be the set of probability measures on the phase-space $\T\times \R^d$
that are invariant under the flow of the differential equation 
\[
\begin{cases}
-\frac{d}{dt} \partial_v L_m(x,x')+\partial_x L_m(x,x')=0\,, 
\\ x(0)=x\,, x'(0)=v.
\end{cases}
\]
That $E_m$ is non-empty is a consequence of the Bogoliubov theorem. Let 
\[ M_m:=\mathrm{arg\, min}\left(E_m\ni \eta\mapsto \iint_{\T\times \R^d} L(x,v)d\eta(x,v)\right).\]
Define $C_m:=\Pi\# M_m$, where $\Pi$ is the projection $\Pi:\T\times \R^d\ni (x,v)\mapsto x$
and $\#$ the pushforward operator of measures.
As $M_m$ is convex, $C_m$ is a convex subset of the set of probability measures. As $m\mapsto \rho\star\theta_m$ is continuous, $m\mapsto C_m$ has a compact graph, and, consequently, a fixed-point. It is then readily checked, as in \cite{CardaliaguetKAM}, that, if we define 
\begin{equation}\label{Eq:WeakKAM1}\overline\lambda:=-\min_{\eta\in E_{\overline m}}\iint_{\T\times \R^d} L_{\overline m}(x,v)d\eta(x,v)\end{equation} then $\overline\lambda$ is the associated ergodic constant of the Hamiltonian, thereby concluding the proof.
\end{proof}

\section{Proof of Theorem \ref{Th:UniquenessMFGCompet}: uniqueness results}
In order to derive the uniqueness of solutions, we will use the Lasry-Lions monotonicity criterion.
From the point of view of presentation, the focus will be on the derivation
of the (strong) monotonicity property of Proposition \ref{Pr:EpsilonMonotone},
and we will show how to use it to derive uniqueness for solutions of \eqref{Eq:MFGCompet} in the second-order case; in this second step we will use the standard computations of \cite{LasryLions1}, and we detail it to underline the way \eqref{Eq:Mono} is used. The first-order case will be treated by adapting the arguments of \cite{CardaliaguetKAM}. Finally, we do not detail the proof of uniqueness in the ergodic setting as it follows by exactly the same type of arguments.

\subsection{Proof of Proposition \ref{Pr:EpsilonMonotone}}
\label{subsec:exist_1}
As already pointed out, Proposition \ref{Pr:EpsilonMonotone}
plays a central role in the uniqueness of solutions to \eqref{Eq:MFGCompet}.
The strategy of proof is novel and consists in one of 
the main contribution of the present work.
It relies on a concavity-type property,
that is related to methods introduced and developed in the context of Nash equilibria
(for a game with finitely many players) in \cite{Mazari_2022}.
\begin{proof}[Proof of Proposition \ref{Pr:EpsilonMonotone}]
    Let $T>0$, fix $m_1\in C^0([0,T],\mathcal{P}(\T))$ and consider the functional:
    \[g: \mathcal C^0([0,T];\mathcal{P}(\T))^2\ni
     m\mapsto \iint_{(0,T)\times \T}(\rho\star \theta_{\e,m_1}-\rho \star\theta_{\e,m})d(m_1-m).\]
    Inequality \eqref{Eq:Mono} will be obtained as
    a consequence of  the strict concavity of $g$.
    
    Define $\e\dot{\theta}_{\e,m}$ and $\e^2\ddot{\theta}_{\e,m}$ 
    as the first and second Gateaux derivatives
    of $ m\mapsto \theta_{\e,m}$ at a given $m$ in the direction $h$. The reason we chose the scaling $\e$ and $\e^2$ is to make the orders of magnitude clearer in subsequent computations.These functions satisfy
 \begin{align}\label{Eq:DotTheta}
     &\begin{cases}
 \partial_t\dot\theta_{\e,m} -\mu\Delta\dot\theta_{\e,m} -\dot\theta_{\e,m} \left(K-\e(\rho\star m)-2\theta_{\e,m} \right)
         =-(\rho\star h)\theta_{\e,m}, 
 \\ \dot\theta_{\e,m}(0,\cdot)=0,\end{cases}
     \\
    \label{Eq:DdotTheta}
     &\begin{cases}
 \partial_t\ddot\theta_{\e,m} -\mu\Delta\ddot\theta_{\e,m} -\ddot\theta_{\e,m} \left(K-\e(\rho\star m)-2\theta_{\e,m} \right)
         =-2(\rho\star h)\dot\theta_{\e,m}-2\dot\theta_{\e,m}^2, 
 \\ \ddot\theta_{\e,m}(0,\cdot)=0.
\end{cases}
 \end{align}
Likewise, the first and second order Gateaux derivative of $g$ at $m$ in the direction $h$ are
\begin{align}\label{Eq:Dotg0}
    &\dot g( m)[h]=\e\iint_{(0,T)\times \T} \dot\theta_{\e,m}(\rho\star(m-m_1))
    +\e\iint_{(0,T)\times \T} (\theta_{\e,m}-\theta_{\e,m_1})(\rho\star h),
\\
    \label{Eq:Ddotg0}
    &\ddot g( m)[h,h]=\e^2\iint_{(0,T)\times \T}\ddot\theta_{\e,m}(\rho\star(m-m_1))
    +2\e\iint_{(0,T)\times \T}\dot\theta_{\e,m}(\rho\star h).
\end{align}
    From the expression of $\dot{g}$, it is easy to check that $m=m_1$
    is a critical point of $g$.
    Now, assume that 
    \begin{multline}
        \label{Eq:g_concave}
        \text{there exists } \e_0,\delta, a>0
        \text{ such that, for $\e\leq\e_0$, and any $T>0$,  }
       \\ \ddot g(m)[h,h]\leq -\frac{a}2\e\iint_{(0,T)\times\T}\dot{\theta}_{\e,m}^2,
    \end{multline}
    where $(\e,a)$ can be chosen uniform in $\nu$ and 
    $\delta$ is the constant in Assumption \ref{H:Theta0}.
    Then, \eqref{Eq:Mono} is obtained using a Taylor integral formula as follows:
    take $h=m-m_1$, $m_{\xi}=\xi m+(1-\xi)m_1$ and $\dot\theta_{\xi}=\dot\theta_{\e,m_{\xi}}$ for $\xi\in[0,1]$,
\begin{align*}
    \iint_{(0,T)\times \T} \Big(\rho \star \theta_{\e,m}-\rho\star\theta_{\e,m_1}\Big)d (m-m_1)
    &=
g(m)-g(m_1)    \\
    &=
    \int_{0}^{1} (1-\xi)\ddot g( m_{\xi})[m-m_1,m-m_1]d\xi.
    \\
    &
    \leq -\frac{a}2\e
    \int_{0}^{1}(1-\xi)\int_0^T\int_\T \dot\theta_{\xi}^2\,dx\,dt\,d\xi,
\end{align*}
    and, switching $m$ and $m_1$, we get
    \begin{equation}
        \iint_{(0,T)\times \T} \Big(\rho \star \theta_{\e,m}-\rho\star\theta_{\e,m_1}\Big)d (m-m_1)
        \leq 
        -\frac{a}2\e
        \int_{0}^{1}\xi\int_0^T\int_\T \dot\theta_{\xi}^2\,dx\,dt\,d\xi,
    \end{equation}
    so that we obtain 
\[  \iint_{(0,T)\times \T} \Big(\rho \star \theta_{\e,m}-\rho\star\theta_{\e,m_1}\Big)d (m-m_1)\leq -a\e\int_0^1\iint_\TT\dot\theta^2_\xi.\]   Estimate \eqref{Eq:Mono} then follows by observing that 
\[-\int_0^1\iint_\TT\dot\theta^2_\xi\leq-\iint_\TT \left(\int_0^1 \dot\theta_\xi\right)^2=-\iint_\TT\left(\theta_{\e,m_1}-\theta_{\e,m}\right)^2.
\]

    Therefore, the remaining of the present proof consists in proving \eqref{Eq:g_concave}.
    Let us first rewrite $\ddot g$ in a more convenient form by using the adjoint state $p_{\e,m}$ defined as the unique solution of the backwards parabolic equation\begin{equation}\label{Eq:Adjoint}
\begin{cases}
-\partial_t p_{\e,m}-\mu\Delta p_{\e,m}-p_{\e,m}(K-\e(\rho\star m)-2\theta_{\e,m})=\rho\star(m_1-m)&\text{ in }(0,T)\times \T\,, 
\\ p_m(T,\cdot)\equiv 0&\text{ in }\T.
\end{cases}
\end{equation}
Using $\ddot\theta_{\e,m}$ as a test function in the weak formulation of \eqref{Eq:Adjoint}
    and $p_{\e,m}$ as a test function in the weak formulation of \eqref{Eq:DdotTheta} we obtain 
\begin{equation}\label{Eq:Ddotg1}
\ddot g( m)[h,h]=2\e\iint_{(0,T)\times \T}
    \dot\theta_{\e,m}
    \left(
    (\rho\star h)(1-\e p_{\e,m})
    -\e p_{\e,m}\dot\theta_{\e,m}\right)\,dt\,dx.
\end{equation}
    Observe that when $\e$ is small both
    $\dot\theta_{\e,m}$ and $p_{\e,m}$ are of order 1. 
    Let us prove that $\ddot g$ is negative for $\e$ small enough.
    Using  \eqref{Eq:DotTheta}, we have
\[ \rho\star h=\frac{-\partial_t\dot\theta_{\e,m}+\mu\Delta\dot\theta_{\e,m}+V_{\e,m}\dot\theta_{\e,m}}{\theta_{\e,m}}
\;\text{ with }\;
V_{\e,m}:=K-\e(\rho\star m)-2\theta_{\e,m}.\]
Defining $\psi_{\e,m}:=\frac{1-\e p_{\e,m}}{\theta_{\e,m}}$, we deduce
\[\ddot g( m)[h,h]
:=2\e\iint_{(0,T)\times \T} \psi_{\e,m}\dot\theta_{\e,m}
\left(-\partial_t\dot\theta_{\e,m}+\mu\Delta\dot\theta_{\e,m} + V_{\e,m}\dot\theta_{\e,m}\right)
-2\e^2\iint_{(0,T)\times \O}p_{\e,m}\dot\theta_{\e,m}^2.\]
Developing and integrating by parts, we obtain
\begin{align*}
    2\iint_{(0,T)\times \T} \psi_{\e,m}\dot\theta_{\e,m}\left(-\partial_t\dot\theta_{\e,m}+\mu\Delta\dot\theta_{\e,m}\right)=&\iint_{(0,T)\times \T}\psi_{\e,m}\left(-\partial_t\dot\theta_{\e,m}^2+\mu\Delta\left(\dot\theta_{\e,m}^2\right)-2\mu |\n \dot\theta_{\e,m}|^2
\right)
    \\=&-2\mu \iint_{(0,T)\times \T }\psi_{\e,m}|\n \dot\theta_{\e,m}|^2-\int_\T \psi_{\e,m}(T,\cdot)\dot\theta_{\e,m}^2(T,\cdot)
    \\&+\iint_{(0,T)\times \T} \dot\theta_{\e,m}^2\left(\partial_t\psi_{\e,m}+\mu\Delta\psi_{\e,m}\right).
\end{align*}
From this and the fact that
$\psi_{\e,m}(T,\cdot)=\frac1{\theta_{\e,m}(T,\cdot)}>0$,
we obtain
\begin{equation*}
  \frac1{\e}  \ddot g( m)[h,h]
    \leq
    \iint_{(0,T)\times \T} \dot\theta_{\e,m}^2\left(\partial_t\psi_{\e,m}+\mu\Delta \psi_{\e,m}
    +2\psi_{\e,m}V_{\e,m}-2\e p_{\e,m}\right)
    -2\mu \iint_{(0,T)\times \T}\psi_{\e,m}|\n \dot\theta_{\e,m}|^2.
\end{equation*}
Introduce, for any $\e\geq 0$ and any $t\geq 0$ the first eigenvalue $\lambda_\e(t)$ defined by the Rayleigh quotient
\begin{equation}\label{Eq:RayleighEpsilon}
\lambda_\e(t):=\min_{u\in W^{1,2}(\T)\,, \int_\T u^2=1}\mu \int_\T \psi_{\e,m}|\n u|^2
    -\frac12\int_\T u^2\left(\partial_t\psi_{\e,m}+\mu\Delta \psi_{\e,m}+2V_{\e,m}\psi_{\e,m}-2\e p_{\e,m}\right).
\end{equation}
To obtain \eqref{Eq:g_concave},
it only remains to prove that,
under the smallness assumptions of the theorem, there exists $\underline\lambda>0$ such that
\begin{equation}\label{Eq:Count}
    \lambda_\e(t)\geq \underline\lambda>0
    \;\text{ for all }\;
    t\in[0,T].
\end{equation}
We thus focus on proving the latter statement.
Adapting \emph{mutatis mutandis} the proof of Proposition~\ref{Pr:UniformRegularityEstimates},
there exists a constant $C>0$ that only depends on $\mu$ but not on $(\nu,T,m_1)$, such that 
\begin{equation}\label{Eq:UniformAdjointEstimates}
\Vert p_{\e,m}\Vert_{L^\infty(0,T;\mathscr C^2(\T))}\leq C.\end{equation}  In particular, combined with Proposition \ref{Pr:UniformRegularityEstimates} this implies that 
\begin{equation}\label{Eq:ConvergencePsi}
\lim_{\e \to 0}\sup_{T\geq 0}\left\Vert \psi_{\e,m}-\frac1{\theta_{\e,m}}\right\Vert_{L^\infty(0,T;\mathscr C^2(\T))}+\left\Vert \partial_t \psi_{\e,m}-\partial_t\frac1\theta_{\e,m} \right\Vert_{L^\infty((0,T)\times \T)}=0.\end{equation}
This implies that, uniformly in $T$ and in $t\in [0,T]$, we have 
\begin{align*} 
    \lim_{\e \to 0} \left(\partial_t\psi_{\e,m}+\mu\Delta \psi_{\e,m}
    +2V_{\e,m}\psi_{\e,m}-2\e p_{\e,m}\right)
    &=\mu\Delta(\overline\theta_K^{-1})
    +K\overline\theta_K^{-1}-2
\\&=3\frac{K-\overline\theta_K}{\overline\theta_K}
    +\frac32\cdot\mu\frac{|\n \overline\theta_K|^2}{\overline\theta_K^3}
    +\left(\frac12\cdot\frac{\mu|\n \overline\theta_K|^2}{\overline\theta_K^3}-2\right)
\\&=:2\Phi_0+\left(\frac12\cdot\frac{\mu|\n \overline\theta_K|^2}{\overline\theta_K^3}-2\right)
\end{align*}
Because of \eqref{Eq:ConvergencePsi} we further deduce that, uniformly in $T$, 
\begin{equation}\label{Eq:CvEigenvalue}
\lim_{\e\to 0}\sup_{t\in[0,T]}\vert\lambda_\e(t)-\lambda_0\vert=0,
\end{equation}
where $\lambda_0$ is the first eigenvalue defined by the Rayleigh quotient
\begin{equation}\label{Eq:Lambda0}
\lambda_0=\min_{u\in W^{1,2}(\T)\,, \int_\T u^2=1}\mu\int_\T \frac1{\overline \theta_K}|\n u|^2-\int_\T u^2\left(\Phi_0+\left(\frac14\cdot\frac{\mu|\n \overline\theta_K|^2}{\overline\theta_K^3}-1\right)\right).
\end{equation}
It is proven in \cite[First analysis of $\overline \xi$, p. 34]{Mazari_2022}
that 
\[\min_{u\in W^{1,2}(\T)\,, \int_\T u^2=1}\mu\int_\T \frac1{\overline\theta_K}|\n u|^2-\int_\T u^2\Phi_0=0.\]
Moreover, in one-dimension,
as observed in \cite[pp. 54-55]{Mazari_2022}, one might use \cite[Estimate 2.2]{BaiHeLi} to have 
\[\frac14\cdot\frac{\mu|\n \overline\theta_K|^2}{\overline\theta_K^3}-1\leq -\frac56.
\]
As demonstrated in \cite{Mazari_2022},
such an estimate also holds in higher dimensions
under the additional assumption that 
$\Vert K-\int_\T K\Vert_{L^\infty(\T)}$ is small enough, whence the assumption on the existence of $\delta>0$
appearing in Assumption \ref{H:Theta0}.
Therefore, we obtain $\lambda_0>0$,
which, combined with the continuity property \eqref{Eq:Lambda0},
yields the existence of a suitable $\e_0$.
Then, \eqref{Eq:Count} holds and we obtain \eqref{Eq:Mono} with $a=\frac12$.
\end{proof}

\subsection{Uniqueness for the MFG system}

We begin with the proof of the second-order case.
\begin{proof}[Proof of Theorem \ref{Th:UniquenessMFGCompet}, second-order case]
Let $\e_0=\e_0(\beta,T)>0$ be given by Proposition \ref{Pr:EpsilonMonotone}. Let us assume that there exists $\e\leq \e_0$ such that there are two distinct solutions  $(u^1,m^1,\theta_{\e,m^1})$  and $(u^2,m^2,\theta_{\e,m^2})$  of \eqref{Eq:MFGCompet}. Define for the sake of readability $\mathscr S(m^i):=\rho\star\theta_{\e,m^i}$,
so that the conclusion of Proposition \ref{Pr:EpsilonMonotone} implies 
\begin{equation}\label{Eq:Pickup} \iint_\TT \left(\mathscr S(m^1)-\mathscr S(m^2)\right)d(m^1 -m^2)\leq 0.\end{equation}
Let $\overline V:=u^1-u^2$ and $\overline m:=m^1-m^2$. Clearly, $(\overline V\,, \overline m)$ satisfy
\begin{equation}\label{Eq:Overlinev}-\partial_t \overline V-\nu \Delta \overline V=\frac{\Vert\n_x u^1\Vert^2}2-\frac{\Vert \n_x u^2\Vert^2}2+\mathscr S(m^1)-\mathscr S(m^2)\end{equation} and 
\begin{equation}\label{Eq:Overlinem}
\partial_t \overline m-\nu \Delta \overline m=-\nabla\cdot(\n_xu^1m^1 )+\nabla\cdot(\n_xu^2m^2 ).
\end{equation}
Using $\overline m$ as a test function in \eqref{Eq:Overlinev} and $\overline V$ as a test function in \eqref{Eq:Overlinem} we obtain (similar to \cite{LasryLions3})
\begin{equation*}
    \iint_\TT \left(\mathscr S(m^1)-\mathscr S(m^2)\right)d(m^1 -m^2)
    =
    \frac12\iint_\TT
    \Vert\nabla_x\overline V\Vert^2(m^1+m^2)
    \geq0.
\end{equation*}
This implies that the inequality in \eqref{Eq:Pickup} is in fact an equality.
By \eqref{Eq:Mono}, this yields $\theta_{\e,m^1}=\theta_{\e,m^2}$, so that $u^1=u^2$ by uniqueness of the
solution to the HJB equation, finally, we get $m^1=m^2$.
\end{proof}
We now deal with the first-order case.
\begin{proof}[Proof of Theorem \ref{Th:UniquenessMFGCompet}, first-order case]
We follow the methodology of  \cite[Theorem 3.3, Proof of uniqueness]{CardaliaguetKAM}. Argue by contradiction and assume that there exists $\e\leq \e_0$ such that there are two distinct solutions  $(u^1,m^1,\theta_{\e,m^1})$  and $(u^2,m^2,\theta_{\e,m^2})$  of \eqref{Eq:MFGCompet}
let $\xi\in \mathscr D(\R^d)$ be a symmetric, non-negative kernel supported in the unit ball with integral 1 and set, for any $\eta>0$, $\xi_\eta:=(1/\eta^d)\xi(\cdot/\eta)$. Define 
\[ m^i_\eta:=\xi_\eta\star m^i\,, W^i_\eta:=\frac{\xi_\eta\star(m^i\n_x u^i)}{m^i_\eta}\,, i=1\,, 2.\]
Observe that this definition is licit as, by Proposition \ref{Pr:UniformRegularityEstimates}, $\n_x u^i$ is continuous $m^i$ a.e., so that   $m^i\n_x u^i$ is a Radon measure. 

We have, by direct computations, 
\[ \partial_t m^i_\eta+\nabla\cdot( m^i_\eta W^i_\eta)=0.\] Now, define 
\[ \overline m_\eta:=m^1_\eta-m^2_\eta,\] so that 
\[ \partial_t \overline m_\eta=-\nabla\cdot(m^1_\eta W_1^\eta)+\nabla\cdot(m_\eta^2 W^2_\eta).\] Multiplying by $\overline u:=u^1-u^2$ and integrating by parts in time we derive
\begin{align*}
\iint_{(0,T)\times \T }(-\partial_t \overline u)\overline m_\eta&=\iint_{(0,T)\times \T}\overline u\left(-\nabla\cdot(m^1_\eta W_1^\eta)+\nabla\cdot(m_\eta^2 W^2_\eta)\right)
\\&=\iint_{(0,T)\times \T}\overline u\left(-\nabla\cdot(\xi_\eta\star(m^1 \n_x u^1))+\nabla\cdot(\xi_\eta\star(m^2\n_x u^2))\right)
\\&=\iint_{(0,T)\times \T }\langle \n_x \overline u,\xi_\eta\star (m^1\n_x u^1)\rangle-\iint_{(0,T)\times \T}\langle \n_x \overline u,\xi_\eta\star(m^2\n_x u^2)\rangle
\\&=\iint_{(0,T)\times \T} m_\eta^1 \langle \n_x \overline u\, ,  \n_x u^1\rangle-\iint_{(0,T)\times \T}m_\eta^2\langle \n_x \overline u, \n_x u^2\rangle
\\&+\iint_{(0,T)\times \T}\langle \n_x\overline u\,, \xi_\eta\star(m^1\n_x u^1)-(\xi_\eta\star m^1)\n_x u^1-\xi_\eta\star(m^2\n_x u^2)+(\xi_\eta\star m^2) \n_x u^2\rangle.
\end{align*}
Define 
\[ R_\eta:=\iint_{(0,T)\times \T}\langle \n_x\overline u\,, \xi_\eta\star(m_1\n_x u^1)-m_1\xi_\eta\star\n_x u^1-\xi_\eta\star(m^2\n_x u^2)+m^2 \xi_\eta\star \n_x u^2\rangle.\]
Using the fact that $\overline u$ is a viscosity solution of \eqref{Eq:Overlinev} with $\nu=0$ we can rewrite the identity above as
\begin{multline*} R_\eta=\iint_{(0,T)\times \T}(\mathscr S(m^1)-\mathscr S(m^2))(m_\eta^1-m_\eta^2)
\\+\iint_{(0,T)\times \T}\overline m_\eta\left(\frac{\Vert \n_x u^1\Vert^2}{2}-\frac{\Vert \n_x u^2\Vert^2}2\right)
\\-\iint_{(0,T)\times \T} m_\eta^1 \langle \n_x \overline u\, ,  \n_x u^1\rangle+\iint_{(0,T)\times \T}m_\eta^2\langle \n_x \overline u, \n_x u^2\rangle.
\end{multline*}
Developing $\overline m_\eta\,, \overline u$ and grouping the terms we obtain 
\begin{multline}
R_\eta=-\iint_{(0,T)\times \T}\left(\frac{\Vert \n_x u^1\Vert^2+\Vert \n_x u^2\Vert^2}2-\langle \n_x u^1,\n_x u^2\rangle\right)(m^1_\eta+m^2_\eta)\\+\iint_\TT \left(\mathscr S(m^1)-\mathscr S(m^2)\right)d(m^1_\eta-m^2_\eta).
\end{multline}
As $m^i\geq 0$, we have $m_\eta^i\geq 0$ for $i=1\,, 2$. Thus, 
\begin{equation}\label{Eq:Etabound} R_\eta\leq \iint_\TT \left(\mathscr S(m^1)-\mathscr S(m^2)\right)d(m^1_\eta-m^2_\eta).\end{equation} Now we have, on the one-hand, 
\begin{equation}\label{Eq:Seta}\iint_\TT \left(\mathscr S(m^1)-\mathscr S(m^2)\right)d(m^1_\eta-m^2_\eta)\underset{\eta\to 0}\rightarrow \iint_\TT \left(\mathscr S(m^1)-\mathscr S(m^2)\right)d(m^1-m^2)\leq 0\end{equation} and, on the other hand,
\begin{equation}\label{Eq:Reta}R_\eta\underset{\eta\to 0}\rightarrow 0.\end{equation} Indeed, we can brutally bound $R_\eta$ as 
\begin{multline*} |R_\eta|\leq C \iint_{||x-y||\leq \eta\,, (x,y)\in \T\times \R^d} |\n_x u^1(x)-\n_x u^1(y)| dm^1(x)dy\\+\iint_{||x-y||\leq \eta\,, (x,y)\in \T\times \R^d} |\n_x u^2(x)-\n_x u^2(y)| dm^2(x)dy.\end{multline*}However, as we observed, $\n_x u^i$ is continuous a.e. on the support of $m^i$, thereby establishing \eqref{Eq:Reta}. Combined with \eqref{Eq:Seta} and \eqref{Eq:Etabound} this entails
\[ \iint_{(0,T)\times \T} (\mathscr S(m^1)-\mathscr S(m^2))d(m^1-m^2)=0.\] We can then conclude in the same way as in the second-order case.
\end{proof}

\subsection{Uniqueness in the ergodic case}
Let us first state the ergodic counterpart to Proposition \ref{Pr:EpsilonMonotone}:
\begin{proposition}\label{Pr:EpsilonMonotoneErgodic}
    Assume \ref{H:K} and \ref{H:Theta0},
    let $\mu>0$.
    There exist $\e_0,\delta,a>0$,
    where $\delta$ is the constant in \ref{H:Theta0}
    such that,
    for any $\e\leq \e_0$,
    the following monotonicity condition holds:
    for $m,m'\in \mathcal P(\T)$,
\begin{equation}
    \label{Eq:Mono2}
    \int_{\T} \Big(\rho \star \overline\theta_{\e,m}-\rho\star\overline\theta_{\e,m'}\Big)d (m-m')
    \leq -\e a\int_\T \left(\overline \theta_{\e,m}-\overline\theta_{\e,m'}\right)^2.
\end{equation}
    \end{proposition}
The proof of this proposition is identical to that of Proposition \ref{Pr:EpsilonMonotone} and is omitted here.
With this proposition at hand, the proof of uniqueness of \eqref{Eq:ErgodicMFG} consists in applying
the standard arguments for the uniqueness of ergodic MFG, as detailed in \cite{CardaliaguetKAM,LionsPapanicolaouVaradhan}.

\section{Proof of Theorem \ref{Th:ErgodicMFG2nd}}
The strategy of proof uses in a crucial way the core ideas of \cite{CardaliaguetKAM}. Let us introduce some notations; we let $\e>0\,,\nu\geq 0$ be fixed and, for any $T>0$ 
\[\theta^T:= \theta_{\e,m}\,, m^T=m\,, u^T:=u\] where $(\theta_{\e,m},m,u)$ is the unique solution of \eqref{Eq:MFGCompet}. We also define the rescaled functions
\[ \Theta^T(s,x):=\theta^T(sT,x)\,, M^T(s,x):=m(sT,x)\,, W^T(s,x):=u^T(sT,x)\text{ for any $(s,x)\in [0,1]\times \T$}.\]
Let $(\overline \lambda, \overline\theta_{\e,\overline m},\overline m,\overline u)$ be the unique solution of the ergodic system \eqref{Eq:ErgodicMFG}.
We start by proving the convergence results on $\Theta^T$.

\begin{proof}[Proof of the $L^2$ convergence of $\Theta^T$]
    Let us define $\overline{\theta}^T$ as the solution of
    \begin{equation}
\begin{cases}
\partial_t\overline \theta^T-\mu\Delta \overline \theta^T-\overline \theta^T\left(K-\e(\rho\star \overline m)-\overline \theta^T\right)=0&\text{ in }[0,T]\times \T\,, 
\\ \overline\theta^T(0,\cdot)=\theta_0&\text{ in }\T.\end{cases}\end{equation}
        Using the change of variables $t=sT$, we get
    \begin{align*}
        \Vert\Theta^T-\overline\theta_{\e,\overline m}\Vert_{L^2((0,1)\times\T)}^2
        &=
\frac1T        \Vert\theta^T-\overline\theta_{\e,\overline m}\Vert_{L^2((0,T)\times\T)}^2
        \\
        &\leq
        \frac1{T}
        \left(\Vert\theta^T-\overline\theta^T\Vert_{L^2((0,T)\times\T)}
        +\Vert\overline\theta^T-\overline\theta_{\e,\overline m}\Vert_{L^2((0,T)\times\T)}\right)^2
        \\
        &=
        \frac2{T}
        \left\Vert
        \theta^T-\overline\theta^T
        \right\Vert_{L^2((0,T)\times\T)}^2
        +\frac2{T}
        \int_0^T
        \Vert\overline\theta^T(t)-\overline\theta_{\e,\overline m}\Vert_{L^2(\T)}^2dt.
    \end{align*}
    However, from Proposition \ref{Pr:EpsilonMonotone} we know that
    \[ \left\Vert
        \theta^T-\overline\theta^T
        \right\Vert_{L^2((0,T)\times\T)}^2\leq -\frac1{a\e} \iint_\TT (\rho\star \theta^T-\rho\star \overline\theta^T)d(m^T-\overline m)\] and, similarly, from \eqref{Eq:EldLongTime}, 
    \[  \int_0^T
        \Vert\overline\theta^T(t)-\overline\theta_{\e,\overline m}\Vert_{L^2(\T)}^2dt\leq     C\] for some constant $T$ independant of $T$, whence
        \[ \Vert\Theta^T-\overline\theta_{\e,\overline m}\Vert_{L^2((0,1)\times\T)}^2\leq  -\frac2{a\e T} \iint_\TT (\rho\star \theta^T-\rho\star \overline\theta^T)d(m^T-\overline m)+\frac{C}{T}.\]
    
It now remains to prove the following lemma:
\begin{lemma}
    \label{Lem:Cardaliaguet}
    Under the assumption of Theorem \ref{Th:ErgodicMFG2nd},
    there exists a constant $C>0$ such that 
    \begin{equation}\label{Eq:Cardaliaguet}
       0\geq \iint_{[0,T]\times \T}\left(\rho\star \theta^T-\rho\star \overline\theta\right) d(m^T-\overline m)\geq - C.
    \end{equation}
\end{lemma}
\begin{proof}[Proof of Lemma \ref{Lem:Cardaliaguet}]
In the second-order case $\nu>0$, it suffices to proceed as in the proof of Theorem \ref{Th:UniquenessMFGCompet} to obtain that 
\begin{multline*}\iint_{[0,T]\times \T}\left(\rho\star \theta^T-\rho\star \overline\theta\right) d(m^T-\overline m)-\iint_{(0,T)\times \T}\frac{(m^T+\overline m)}2\Vert \n_x u^T-\n_x \overline u\Vert^2\\=\int_\T \overline u (m(T,\cdot)-\overline m)+\int_\T (u^T(0,\cdot)-\overline u)(m^T-\overline m).
\end{multline*}
Observe that  $\overline u\in L^\infty(\T)$. Furthermore, since $\int_\T m^T=\int_\T \overline m$, for any $\bar x\in \T$,
    \[ \int_\T u^T (0,\cdot)(m^T-\overline m)=\int_\T (u^T(0,\cdot)-u^T(0,\bar x))(m^T-\overline m)\leq2\Vert \overline m \Vert_{L^1(\T)} \Vert \n_x u^T\Vert_{L^\infty([0,T]\times\T)}.\]
    From this and Proposition \ref{Pr:UniformRegularityEstimates}, there exists a constant $C_0$ such that 
    \begin{equation*}\iint_{[0,T]\times \T}\left(\rho\star \theta^T-\rho\star \overline\theta\right) d(m^T-\overline m)-\iint_{(0,T)\times \T}\frac{(m^T+\overline m)}2\Vert \n_x u^T-\n_x \overline u\Vert^2\geq -C_0.
\end{equation*}
    This result holds in the first order case, using a similar approach as  the proof of Theorem \ref{Th:UniquenessMFGCompet}.
    \end{proof}
We thus deduce that 
\[\Vert\Theta^T-\overline\theta_{\e,\overline m}\Vert_{L^2((0,1)\times\T)}^2\leq\frac{C'}T\] for some constant $C'$ (that also depends on $\e$).
\end{proof}

\begin{proof}[Proof of the convergence of $W^T$]
    We can now conclude exactly as in \cite{CardaliaguetKAM}.
    Define $w^T:=\overline u +T\overline \lambda(1-s)$, which solves
\[-\frac{\partial_s w^T}T-\nu \Delta w^T+\frac{\Vert \n_x w^T\Vert^2}2=\rho \star \overline\theta_{\e,\overline m}\] while $W^T$ solves the same Hamilton-Jacobi equation with the right hand side replaced with $\rho\star \Theta^T$. By standard Hamilton-Jacobi estimates we obtain 
    \[ \Vert w^T(s,\cdot)-W^T(s,\cdot)\Vert_{L^\infty(\T)}\leq \Vert w^T(1,\cdot)-W^T(1,\cdot)\Vert_{L^\infty(\T)}+ T\int_0^1 \Vert \rho\star(\Theta^T-\overline\theta_{\e,\overline m})\Vert_{L^\infty(\T)}.\]
    Then, using
    the $L^2$-convergence of $\Theta^T$ and the
    Cauchy-Schwarz inequality, we get
\[ \Vert W^T/T-\overline \lambda(1-s)\Vert_{L^\infty}\leq \frac{C}T
    +\Vert\rho\Vert_{L^2(\T)}
    \Vert \Theta^T-\overline\theta_{\e,\overline m}\Vert_{L^2([0,T]\times\T)}
   \leq \frac{C'}{\sqrt{T}}\]
\end{proof}

\section{Proof of Theorem \ref{Th:Ergodic1D}}
Recall that Theorem \ref{Th:Ergodic1D} only holds in
dimension one with additional assumptions on the function $K$.
We will first give an explicit construction of solutions, and then check that these solutions are in fact unique. 
\paragraph{Construction of explicit solutions}
\begin{lemma}\label{Le:Rearrangement}
    Under the assumptions of Theorem \ref{Th:Ergodic1D},
    for $y\in (0,0.5)$, define $\theta_y$ as the solution of 
\begin{equation}\label{Eq:Aux}\begin{cases}
-\mu\theta''_y=\theta_y(K-\theta_y)&\text{ in }(y,0.5)\,, 
\\ \theta_y'(y)=\theta_y'(0.5)=0\,, \theta \geq 0\,, \theta_y(y)>0.\end{cases}\end{equation}
    Then:
    \begin{enumerate}
    \item $\theta_y$ decreasing on $(y,0.5)$ and satisfies
    $K(0.5)\leq\theta_y\leq K(y)$.
  \item $y\mapsto \th_y(y)$ is decreasing and $\lim_{y\to 0.5^-}\theta_y(y)=K(0.5)$.
  \end{enumerate}
\end{lemma}
\begin{proof}[Proof of Lemma \ref{Le:Rearrangement}]
    For a fixed $y\in(0,0.5)$, the existence of $\theta_y$
    and its monotonicity follow from rearrangement inequalities.   The bounds on $\theta_y$ are obtained using the maximum principle.
  
    Let us  prove that $y\mapsto\theta_y(y)$ is decreasing.
    Assume by contradiction that there exist $0<y_1<y_2< 0.5$ such that
    $\theta_{y_1}(y_1)\leq\theta_{y_2}(y_2)$.
As $\theta_{y_1}$ is decreasing, $\p(y_2)<0$,
    where $\p:=\theta_{y_1}-\theta_{y_2}$
    satisfies
\begin{equation}\label{Eq:phi} -\mu\p''-\p(K-(\theta_{y_1}+\theta_{y_2}))=0\text{ in }(y_2,0.5)\,, \p'(y_2)<0\,, \p'(0.5)=0.\end{equation}
    Let $\p_-$ be the negative part of $\p$;
    we have $\p_-'(y_2)>0$ and $\p_-(y_2)>0$.
    Let us multiply \eqref{Eq:phi} by $\p_-$ and integrate by parts, we obtain
\[\mu\int_{y_2}^1 (\p_-')^2-\int_{y_2}^1 \p_-^2 (K-(\theta_{y_1}+\theta_{y_2}))=-\mu\p_-'(y_2)\p_-(y_2)<0. \]
    This is a contradiction with the fact that
    $-\mu\partial_{xx}^2-(K-(\theta_{y_1}+\theta_{y_2}))$
    admits a positive first eigenvalue.
\end{proof}

\begin{proof}[Proof of the existence of solutions to \eqref{Eq:Ergodic1D}]
Let $\e\in (0,\int_{\mathbb{T}}K)$. There are two cases to be handled:    \begin{itemize}
        \item
            If $\e\geq\int_{\mathbb{T}} K-K(0.5)$,
            define 
            \begin{equation}
                \label{Eq:ExplicitSol2}
    \overline \theta:=\int_{\mathbb{T}}K-\e,
    \hspace*{0.5cm}
    \overline \lambda:=\int_{\mathbb{T}}K-\e,
    \hspace*{0.5cm}
    \overline m:=1 +\e^{-1}(K-\int_{\mathbb{T}}K)
    \hspace*{0.5cm}
    \text{ and }
    \hspace*{0.5cm}
    \overline u:= 0.
            \end{equation}
        \item
            If $\e<\int_{\mathbb{T}} K-K(0.5)$,
            by the mean-value theorem,
            there exists a unique $y\in(0,0.5)$ such that
            $\e=2\int_0^yK-2y\theta_y(y)$:
            indeed, it is easy to check that
            $y\mapsto\int_0^yK-y\theta_y(y)$ is increasing
            by taking its derivative and using Lemma \ref{Le:Rearrangement}.
            Then define, for $x\in(-0.5,0.5)$,
            \begin{equation}
                \label{Eq:ExplicitSol}
                \begin{aligned}
                    \overline\theta(x)&:=\theta_y(y)\mathds 1_{(0,y)}(|x|)+\theta_y(|x|)\mathds 1_{(y,1)}(|x|),
    \hspace*{2cm}
    \overline \lambda:=\theta_y(y),
    \\
    \overline m(x)&:=\e^{-1}(K-\theta_y(y))\mathds 1_{(0,y)}(|x|)
    \hspace*{0.5cm}
    \text{and}
    \hspace*{0.5cm}
    \overline u(x):=\overline\lambda 
    -\int_y^{\sup(|x|,y)} \sqrt{2(\overline\lambda-\overline \theta)}.
            \end{aligned}
        \end{equation}
    \end{itemize}
    In both cases,
direct computations show that $(\overline\lambda,\overline\theta,\overline m,\overline u)$ solves \eqref{Eq:Ergodic1D}
(we use the fact that $\overline u$ is $\mathscr C^1$, so that it is indeed the viscosity solution of the Hamilton-Jacobi equation). 

\end{proof}

\begin{remark}[Regarding the construction of \eqref{Eq:ExplicitSol}]
    To get an intuition on this construction, recall that from the weak KAM formula \eqref{Eq:WeakKAM1},
    $\overline \theta$ should be constant in $\mathrm{supp}(\overline m)$ and, in fact 
\begin{equation}\label{Eq:WeakKAM1D}
\mathrm{supp}(m)\subset \{ \overline \theta=\Vert \overline \theta\Vert_{L^\infty(\mathbb{T})}\}.
\end{equation} The expression of $\overline u$ is standard, see \cite{LionsPapanicolaouVaradhan}. \end{remark}
\paragraph{Obtaining the uniqueness of solutions}
In this paragraph, we establish that the solutions we just constructed are the unique solutions of \eqref{Eq:Ergodic1D}.
\begin{lemma}\label{Le:NoAtoms}
    Take $(\overline \lambda,\overline \theta,\overline m,\overline u)$ a solution of \eqref{Eq:Ergodic1D}.
    The measure $\overline m$ has no atoms.
\end{lemma}
\begin{proof}[Proof of Lemma \ref{Le:NoAtoms}]
Argue by contradiction and assume that $\overline m$ has atoms. Let $X_{\mathrm{sing}}$ be the set of atoms of $\overline m$. We first observe that $X_{\mathrm{sing}}$ necessarily has an isolated point. Indeed,  should that not be the case, $X_{\mathrm{sing}}$ would  be dense in an open subset $\omega\neq \emptyset$ of $(-0.5,0.5)$. It follows from \eqref{Eq:WeakKAM1D} that $\overline \theta$ is constant in $\overline \omega$. As $\overline \theta$ is continuous, it is then constant in $\overline \omega$, whence $K-\e\overline m-\overline \theta=0$ in $\overline \omega$, whence $m$ has no atom in $\omega$, a contradiction. Now, let us show that $X_{\mathrm{sing}}=\emptyset$. If this is not the case, let $x^*$ be an isolated atom of $m$.  In particular, $\overline \theta$ is $\mathscr C^1$ in $(x^*-\e,x^*)\cup (x^*,x^*+\e)$. Indeed, in the sense of distributions, 
\[ -\mu\overline \theta'(x)=\int_0^x \overline \theta(K-m-\overline\theta),\] and as $m^*$ has no atoms in $(x^*-\e,x^*)\cup (x^*,x^*+\e)$, the conclusion follows.

Now, as $K\,, \overline \theta$ are continuous, $-\mu \partial^2_{xx} \overline \theta$ has an atom at $x^*$. Adopting the notation $[f](x):=\lim_{\e\to 0} (f(x+\e)-f(x-\e))$ we obtain 
\[ \left[\overline\theta'\right](x^*)=\frac{\overline \theta(x^*)}{\mu}\overline m(\{x^*\})>0.\] However, this is in contradiction with $\overline \theta(x^*)=\left\Vert \overline \theta\right\Vert_{L^\infty}$.
\end{proof}
As a consequence of Lemma \ref{Le:NoAtoms}, $\overline\theta$ is $\mathscr C^1$ in $\mathbb T$. 
\begin{lemma}\label{Le:SuppM}
There exists $y\in (0,1)$ such that $\mathrm{supp}(\overline m)=[-y,y]$.
\end{lemma}
\begin{proof}[Proof of Lemma \ref{Le:SuppM}]
We begin by showing that $\mathrm{supp}(\overline m)\cap [0,0.5]$ is connected. Should this not be the case, there exists $a<b$ such that $(a,b)\subset \mathrm{supp}(\overline m)^c$, $a,b\in \mathrm{supp}(\overline m)\cap [0,0.5] $.
    By \eqref{Eq:WeakKAM1D}, we have
    $\overline \theta(a)=\overline \theta(b)=\Vert \overline \theta\Vert_{L^\infty(\mathbb{T})}$,
    so that
    $\overline \theta'(a)=\overline \theta'(b)=0$.
    Observe that Lemma \ref{Le:Rearrangement} holds if the end point
    of the interval is $b$ instead of $0.5$.
    Therefore, we get $\overline \theta(b)<\overline \theta(a)$, a contradiction.

    Let us now show that $0\in \mathrm{supp}(m)$:
    if $a=\inf(\mathrm{supp}(m)\cap [0,0.5])>0$, then,
    using a similar argument as above, we get
    $\overline \theta(0)>\overline \theta(a)=\Vert\overline \theta\Vert_{L^\infty}$, a contradiction.
\end{proof}
Finally, let us conclude the proof of Theorem \ref{Th:Ergodic1D}.
\begin{proof}[End of the proof of Theorem \ref{Th:Ergodic1D}]
We proved that the support of $m$ is of the form $[-y,y]$ for some $y\in(0,0.5]$
and that $\overline\theta$ is constant on that support, denote by $c\geq0$ that constant.
Assume $c>0$, using the reaction-diffusion equation satisfied by $\overline{\theta}$,
we get $K-\e m-c=0$ which implies $\int K>\e$ by integration.
We just proved that $\e\geq \int K$ implies $\theta\equiv0$.

Now assume that $y<0.5$, by integrating on $[-y,y]$ we obtain
$$
    \e=
    \int_{-y}^yK-2yc
    =
    2\left(\int_{0}^yK-y\theta_y(y)\right)
    <
    \int_{\mathbb{T}}K-K(0.5),
$$
where the inequality comes from 
$y\mapsto\int_0^yK-y\theta_y(y)$ being increasing
as used in the construction of \eqref{Eq:ExplicitSol}.
Consequently,
if $\int_{\mathbb{T}}K-K(0.5)\leq \e <\int_{\mathbb{T}}K$
then $y=0.5$ which implies that $\overline\th\equiv\int_{\mathbb{T}}K-\e$.
This and $K-\e m-\overline\th=0$ imply
$\overline m=1+\e^{-1}(K-\int_{\mathbb{T}}K)$.
We recover \eqref{Eq:ExplicitSol2} using \eqref{Eq:Ergodic1D}.

The only case that remains to be treated is
$\e<\int_{\mathbb{T}}K-K(0.5)$.
Using a similar argument, we can prove that
it implies $y<0.5$, so that $\overline{\theta}$ 
coincides with $\theta_y$ from Lemma \eqref{Le:Rearrangement}.
We conclude that the solution can only be given
by \eqref{Eq:ExplicitSol} as in the previous paragraph.
\end{proof}

\bibliographystyle{abbrv}
\bibliography{BiblioNash,MFGbiblio}

\end{document}